\documentclass[11pt,a4paper,leqno]{amsart}

\usepackage{common}

\usepackage{amssymb}
\usepackage{graphicx}

\title{Curved commutators in the plane}

\author{Kangwei Li}
\author{Henri Martikainen}
\author{Tuomas Oikari}

\address[K.L.]{School of Mathematical Sciences, Zhejiang Normal University, Jinhua 321004, China}
\email{kangwei.li@zjnu.edu.cn}

\address[H.M.]{Department of Mathematics and Statistics, Washington University
in St. Louis, 1 Brookings Drive, St. Louis, MO 63130, USA}
\email{henri@wustl.edu}

\address[T.O.]{Departament de Matem\`atiques, Universitat Aut\`onoma de Barcelona,
	Edifici C Facultat de Ci\`encies, 08193 Bellaterra (Barcelona), Catalonia}
\email{tuomas.oikari@gmail.com
}

\makeatletter
\@namedef{subjclassname@2010}{
    \textup{2010} Mathematics Subject Classification}
\makeatother


\begin{document}

\allowdisplaybreaks

\begin{abstract}
	{
		We complete the $L^p$ boundedness theory of commutators of Hilbert transforms
		along monomial curves by providing the previously missing lower bounds. This optimal result
		now covers all monomial curves while the previous result assumed the curve to intersect adjacent quadrants of the plane.
		We also develop, under a qualitative $\BMO$ assumption of the symbol,
		the corresponding quantitative lower bound in the context of
		curves with non-vanishing torsion.
	}
\end{abstract}

\maketitle

\section{Introduction}

We push forward the theory of commutators of singular integrals along curves
by the third named author \cite{OIKARI-PARABOLIC} and Bongers--Guo--Li--Wick
\cite{BGLW}. First of all, commutators of singular integral operators (SIOs) $T$ and functions $b$
have the general form $[b, T]f := bTf - T(bf)$. The theory and applications
of commutator estimates with standard SIOs,
such as, the Riesz transforms $R_j f(x) := \int \frac{x_j-y_j}{|x-y|^{d+1}} f(y)\ud y$,
are extremely rich and well-developed. This includes everything from classical contributions, such as, \cite{CRW}
to more recent state-of-the-art characterizations \cite{HyCom}. A fundamental problem has been to \emph{characterize}
when $[b, T]$ maps $L^p\to L^p$ in terms of a suitable function space of the symbol, often $\BMO$.

Our setting is significantly different in that we consider
singular integrals along \emph{curves} $\gamma$, in particular, the Hilbert transform $H_{\gamma}$
along a monomial curve, where
$$
	H_\gamma f(x) := \int_{\R} f(x-\gamma(t)) \frac{\ud t}{t}.
$$
Here the monomial curve is given by
$$
	\gamma(t) =
	\begin{cases}
		(\epsilon_1 |t|^{\beta_1}, \epsilon_2 |t|^{\beta_2}) & t>0,     \\
		(\delta_1 |t|^{\beta_1}, \delta_2 |t|^{\beta_2})     & t \le 0,
	\end{cases}
$$
$\beta_2>\beta_1>0$, $\epsilon_i, \delta_i = \pm 1$ with $\epsilon_j \ne \delta_j$ for
at least one $j \in \{1,2\}$.
In this curved setting all commutator estimates are very recent and
certainly not yet fully developed.

First, in \cite{BGLW} upper bounds of the form
\begin{align*}
	\| [b, H_\gamma] \|_{L^p\to L^p} & \lesssim
	\sup_{\substack{Q = I \times J                                           \\ \ell(I)^{1/\beta_1} = \ell(J)^{1/\beta_2}}}
	\frac 1{|Q|}\int_Q|b- \langle b\rangle_Q|\ =: \|b\|_{\BMO_\gamma(\R^2)}, \\
	\langle b \rangle_Q              & := \frac{1}{|Q|} \int_Q b,
\end{align*}
were proved, giving sufficiency in terms of a $\gamma$-adapted
$\BMO$ space -- for instance, the parabolic $\BMO$ space when $\gamma(t) = (t, t^2)$.
This was based on sparse domination in this setting by Cladek and Ou \cite{ClOu}.
The corresponding lower bound, or necessity, was left missing and was later proved in \cite{OIKARI-PARABOLIC}
using a completely new curved adaptation of the recently very successful approximate weak factorization
method of \cite{HyCom}. However, even \cite{OIKARI-PARABOLIC} did not deal with all of the monomial curves.
The argument only worked for those
monomial curves that intersect adjacent quadrants of the plane -- that is $\epsilon_j = \delta_j$
for exactly one $j \in \{1,2\}$. Moreover,
\cite{OIKARI-PARABOLIC} only dealt with the parabolic case $\gamma(t) = (t, t^2)$ in detail.

We provide a new argument
that, importantly, is able to deal with curves intersecting opposite quadrants.
In addition, we provide full details for general $\beta$ parameters.
To fully complete the theory, the case of interest to us is
$\gamma(t)= (|t|^{\beta_1} \sign t, |t|^{\beta_2} \sign t)$.
The following theorem, our first main result, thus completes the $L^p \to L^p$ commutator boundedness
theory for the Hilbert transform along monomial curves.

\begin{thm}\label{thm:main}
	Let $b\in L_{\loc}^1(\R^2;\C)$ and $H_\gamma f(x) := \int_{\R} f(x-\gamma(t)) \frac{\ud t}{t}$
	be the Hilbert transform along the curve $$\gamma(t)= (|t|^{\beta_1} \sign t, |t|^{\beta_2} \sign t), $$ where
	$\beta_2>\beta_1>0$. Let $1 < p < \infty$ and suppose that $[b, H_\gamma]$ is a bounded operator on $L^p.$
	Then $b\in \BMO_\gamma(\R^2)$ -- in fact, we have the quantitative commutator lower bound
	$$
		\|b\|_{\BMO_\gamma(\R^2)}
		\lesssim \| [b, H_\gamma] \|_{L^p\to L^p}.
	$$
\end{thm}

Approximate weak factorization (awf) arguments rely on
setting up suitable translations of cubes that respect the structure of the underlying
singular integral, and this is significantly harder in curved settings.
Moreover, the proof presented here is different and more complicated
than the one provided in \cite{OIKARI-PARABOLIC}.
In particular, we introduce an additional parameter into the construction of the
underlying geometry that amounts to a geometric scale jump between two successive factorizations.
The previous curved proof from \cite{OIKARI-PARABOLIC}
did not have this aspect, instead exploiting the symmetry of the graph of $\gamma$
across the $x_2$-axis.

We have also written the proofs of all of our auxiliary results so that they are
purely analytic and formally easier to check, as compared to relying on geometric facts; in particular,
even when the claim is similar, our proofs are different in style from those given in \cite{OIKARI-PARABOLIC}.
We handle the general $\beta$ parameters with a change of variables
and a careful argument adapted to the case $\beta_1 = 1$ and $\beta_2 > 1$.
The change of variables is not complicated but at the very least constitutes a major simplification -- the proof, at least as written now, would not otherwise
work with general $\beta$.

We also study commutators of Hilbert transforms along boundedly supported
curves with non-vanishing torsion.
The non-vanishing torsion condition ensures that the vectors
$\gamma'(0)$ and $\gamma''(0)$ are linearly independent.
While the curve can in this setting lack a true nonisotropic dilation structure,
the assumptions nevertheless allow one to develop the corresponding
theory for the truncated Hilbert transform
$$
	f \mapsto \int_{-1}^1 f(x-\gamma(t)) \frac{\ud t}{t}.
$$
Results of this flavor, but only regarding commutator upper bounds,
have been relatively briefly studied at least in \cites{ClOu, BGLW}.
In section \ref{sec:nonvan} we carefully develop the corresponding
commutator lower bound theory by bootstrapping it from the parabolic
results of \cite{OIKARI-PARABOLIC}. However, in this non-vanishing torsion setting
our quantitative results currently require us to impose an \emph{a priori}
qualitative $\BMO$ assumption on the symbol $b$. The technical details
of this settings are subtle and interesting.

\subsection*{Acknowledgements}
K. Li was supported by the National Natural Science Foundation of China through
project numbers 12222114 and 12001400. This material is based upon work supported
by the National Science Foundation (NSF) under Grant No. 2247234 (H. Martikainen).
H.M. was, in addition, supported by the Simons Foundation through MP-TSM-00002361
(travel support for mathematicians). T.O. was supported by the Finnish Academy of Science and Letters.

\subsection*{Statements and Declarations}
\subsubsection*{Data availability statement} There is no data.
\subsubsection*{Conflict of interest statement} There is no conflict of
interest.

\section{Preliminaries}\label{sect:pain}

\subsection*{Change of variables}

We study the following monomial curve
\begin{equation*}
	\gamma(t)=
	( |t|^{\beta_1} \sign t, |t|^{\beta_2} \sign t),\quad \beta_2>\beta_1>0.
\end{equation*}
The complementary case of curves $\gamma$ satisfying that their graph ``intersects adjacent quadrants''
was dealt with in \cite{OIKARI-PARABOLIC}; although, the argument was explicitly given only
in the parabolic case $\gamma(t) = (t, t^2)$. In fact, many parts of the proofs become, at the very least,
significantly more laborious to run through directly with these general
$\beta = (\beta_1,\beta_2)$ and thus, in retrospect, the
following simple change of variables becomes very useful:
\begin{equation}\label{eq:CoVez}
	\begin{split}
		H_\gamma f(x) & =\int_{-\infty}^0 f(x-\gamma(t) )\frac{\ud t } t +\int_0^\infty f(x-\gamma(t) )\frac{\ud t } t          \\
		              & = \int_{-\infty}^0 f(x+ (|t|^{\beta_1}, |t|^{\beta_2}) )\frac{\ud t } t +\int_0^{\infty}
		f(x- (|t|^{\beta_1}, |t|^{\beta_2}) )\frac{\ud t } t                                                                    \\
		              & = \int_{\infty}^0 f(x+ (|u|, |u|^{\beta_2/{\beta_1}}) )\frac{\ud (-u^{1/{\beta_1}})} {-u^{1/{\beta_1}}}
		+\int_0^{\infty} f(x- (|u|, |u|^{\beta_2/{\beta_1}}) )\frac{\ud (u^{1/{\beta_1}}) } {u^{1/{\beta_1}}}                   \\
		              & = \frac 1 {\beta_1}\Big( \int_{\infty}^0 f(x+ (|u|, |u|^{\beta_2/{\beta_1}}) )\frac{\ud u}u
		+ \int_0^{\infty} f(x- (|u|, |u|^{\beta_2/{\beta_1}}) )\frac{\ud u}u\Big)                                               \\
		              & =  \frac 1 {\beta_1} H_{\tilde \gamma} f(x),
	\end{split}
\end{equation}
where $\tilde \gamma (t)= ( |t| \sign t, |t|^{\beta_2/{\beta_1}} \sign t)$.
Hence we will assume in the future that $\beta_1=1$ and $\beta_2=:\beta>1$.

\subsection{Geometry behind the factorization}\label{sec:geom} Next, we define several auxiliary sets that we will be working with for the rest of the article, see the Figure \ref{fig:geomfull} below for a sketch.
Before the actual construction of those sets, we motivate the configuration with a brief
explanation based on the picture.
Denote
$
	\phi(a,B) := \{a+\gamma(t): t\in \mathbb{R}\} \cap B
$, for any point $a$ and any set $B.$
Consider a point $z\in P$ and $\phi(z,W_1)$ (the intersection of the blue curve with the set $W_1$) and for each $y\in \phi(z,W_1)$ the set $\phi(y,Q)$ (the intersection of the orange curve beginning from the point $y$ with $Q$). Then the picture indicates that if we take the union over all these segments there holds that
\begin{align}\label{eq:covering}
	Q = \bigcup_{y\in \phi(z,W_1)}\phi(y,Q),\qquad \text{ for all points } z\in P.
\end{align}
A heuristic reason why this should hold is as follows.
Consider the cube $Q$ and its left top corner $v_{lt}$ and its right bottom corner $v_{rb}$ (indicated in the picture) and the sets $\phi(v_{lt},W_1)$ and $\phi(v_{rb},W_1)$ (intersection of $W_1$ with the red curves that begin from those points). Then we have non-empty intersections $y_t := \phi(z,W_1)\cap \phi(v_{lt},W_1)$ and $y_b := \phi(z,W_1)\cap \phi(v_{rb},W_1).$ This means that
$1)$ there holds that $v_{lt} = \phi(y_t,Q)$ and $v_{rb} = \phi(y_b,Q) $
and, hence, $2)$ that the cube $Q$ lies in between the orange curves that begin from these points. That this continues to hold for each point $z\in P$ follows from the fact that we choose $A_2\gg A_1$ in the construction of the sets, see below.
Now, we can see that if we take the union over all the sets $\phi(y,Q)$ (indicated in orange) these fully cover the cube $Q.$

The identity \eqref{eq:covering} is used in Section \ref{sect:AWF} when we perform the approximate weak factorization argument to prove Proposition \ref{prop:awf} and without it the properties on the line \eqref{eq:prop:AWF1} would fail.

\begin{figure}[h]
	\centering
	\includegraphics[scale=0.65]{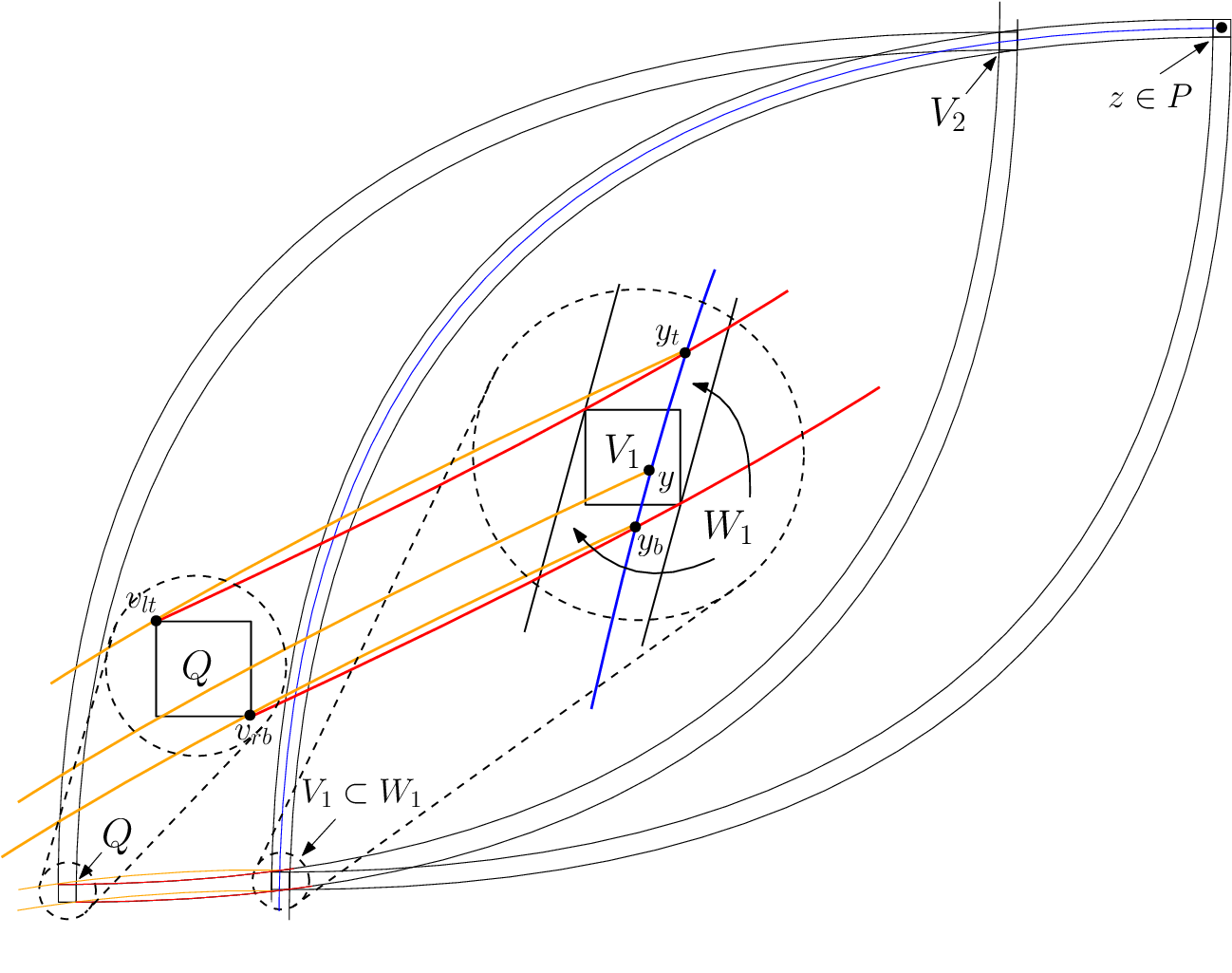}
	\caption{The geometric configuration}\label{fig:geomfull}
\end{figure}

We then carefully define these sets.
Fix a closed rectangle $Q=I\times J\in \mathcal R_\gamma$ -- this means that
$\ell(J)=\ell(I)^{\beta}$. Let
$$
	A_1>10,\qquad A_2 := \big(3^{\beta / (\beta-1)} + 10 \big)A_1 =: C_{\beta}A_1.
$$
Define
\[
	V_1= Q+ \gamma(A_1\ell(I)),\quad P=V_1+\gamma (A_2\ell(I)).
\]
We also define
$$
	V_2= P+\gamma(-A_1\ell(I)).
$$
It then holds that
\[
	Q= V_2+ \gamma(-A_2\ell(I)).
\]
Next, we define
\begin{align*}
	\widetilde Q= \{ Q+\gamma(t): t\ge 0\},\quad \widetilde P= \{ P+\gamma(s): s\le 0\}.
\end{align*}
Note that
\[
	V_1 = Q+ \gamma(A_1\ell(I)) = P+ \gamma(-A_2\ell(I))
\]
and
\[
	V_2= Q+ \gamma(A_2\ell(I))= P+\gamma(-A_1\ell(I)),
\]
and so $V_1 \cup V_2\subset \widetilde Q\cap  \widetilde P$.
Finally, we consider the following subset of $\widetilde Q\cap  \widetilde P$:
\begin{align*}
	W_1 := \big\{Q+\gamma(t) & \colon  (A_1-2)\ell(I)<t < (A_1+2)\ell(I)\big\}                                \\
	                         & \cap\big \{P + \gamma(s) \colon (A_2-3)\ell(I) < |s| < (A_2+3)\ell(I)  \big\}.
\end{align*}

We next study the set $W_1$ (but will define a similar set $W_2$ later).
We start with the following trivial observation.
\begin{lem}\label{lem:sizew}
	We have
	$$
		|W_1| \sim_{A_1} |Q|.
	$$
\end{lem}
\begin{proof}
	First, notice that $V_1\subset W_1$. Second, we have
	\[
		W_1 \subset [u_1, u_1 + (A_1+3) \ell(I)]
		\times [u_2, u_2 + ((A_1+2)^{\beta}+1) \ell(I)^\beta],
	\]
	where $u=(u_1, u_2)$ is the left bottom corner of $Q$.
\end{proof}

\begin{lem}\label{lem:11}
	If $t \le |s|$ and $x + \gamma(t) = y + \gamma(s)$ for some $x \in Q$ and $y \in P$,
	then we must have $t \in ((A_1-2)\ell(I), (A_1+2)\ell(I))$ and $|s| \in ((A_2-3)\ell(I), (A_2+3)\ell(I))$.

	Moreover, for any $x\in Q$ and $y\in P$, there is a unique pair $(t,s)\in \R_+ \times \R_-$ with
	\[
		x+\gamma(t) = y+\gamma(s)\in W_1.
	\]
\end{lem}
\begin{proof}
	Fix $x \in Q$ and $y \in P$. By the definition of $P$, we have
	$y = x'+\gamma(A_1\ell(I))+\gamma(A_2\ell(I))$ for some $x' \in Q$.
	We derive what $x + \gamma(t) = y + \gamma(s)$ for $t \le |s|$ requires.
	For this to hold, we must have
	$$
		x+\gamma(t)=x'+\gamma(A_1\ell(I))+\gamma(A_2\ell(I))+\gamma(s).
	$$
	Componentwise this reads
	\begin{eqnarray}
		x_1+t &=& x_1'+A_1\ell(I)+A_2\ell(I)- |s|, \label{eq:e1}\\
		x_2+t^\beta&=& x_2'+(A_1\ell(I))^\beta+(A_2\ell(I))^\beta- |s|^\beta. \label{eq:e2}
	\end{eqnarray}
	Now, these equations give that
	\begin{equation}\label{eq:e9}
		(A_1^\beta+ A_2^\beta)\ell(I)^\beta+ x_2'-x_2= t^\beta+ \big[ (A_1+A_2) \ell(I)+x_1'-x_1-t \big]^\beta.
	\end{equation}
	A requirement for $t \le |s|$ is obtained by \eqref{eq:e1}:
	$$
		t = \big[ (A_1+A_2) \ell(I)+x_1'-x_1\big] - |s| \le \big[ (A_1+A_2) \ell(I)+x_1'-x_1\big] - t,
	$$
	and so
	\begin{equation}\label{eq:e03}
		t\le \big[ (A_1+A_2) \ell(I)+x_1'-x_1 \big]/2.
	\end{equation}
	In this range \eqref{eq:e03} of $t$ we note that
	$$
		\eta(t) := t^\beta+ \big[ (A_1+A_2) \ell(I)+x_1'-x_1-t \big]^\beta = {\rm{RHS}} \eqref{eq:e9}
	$$
	is a decreasing function of $t$. Indeed, a function $t \mapsto t^{\beta} + (\rho-t)^{\beta}$
	is decreasing if $t^{\beta-1} \le (\rho-t)^{\beta-1}$, which is clearly true if $t \le \rho/2$.
	Also note for future use that $(A_1+2)\ell(I) \le (A_1+A_2-1) \ell(I) /2 \le
		\big[ (A_1+A_2) \ell(I)+x_1'-x_1 \big]/2$, so $t=(A_1 + 2)\ell(I)$ is in this range with any $x, x'$.

	Using $|x_1-x_1'| \le \ell(I)$ and $|x_2-x_2'| \le \ell(J) = \ell(I)^{\beta}$,
	we see that for $(A_1 + 2)\ell(I) \le t \le \big[ (A_1+A_2) \ell(I)+x_1'-x_1 \big]/2$ we have
	\begin{align*}
		\eta(t) \le \eta\big((A_1+2)\ell(I)\big) & \le \ell(I)^\beta \Big((A_1+2)^\beta +(A_2-1)^\beta \Big)              \\
		                                         & <\ell(I)^\beta  (A_1^\beta+ A_2^\beta-1 )\le {\rm{LHS}} \eqref{eq:e9},
	\end{align*}
	which is a contradiction, thus for  \eqref{eq:e9} to hold $t< (A_1+2)\ell(I)$ necessarily.
	We used above that $(A_1+2)^\beta +(A_2-1)^\beta<  A_1^\beta+ A_2^\beta-1$ is true with our explicit
	choice of $A_2$. Indeed -- to this end, define $h(t) = t^{\beta}$ and then by the
	mean value theorem we have for $\xi_1 \in (A_1,  A_1 + 2)$ and $\xi_2 \in (A_2-1, A_2)$ that
	\begin{align*}
		A_2^\beta- (A_2-1)^\beta & = h(A_2) - h(A_2-1)                                          \\
		                         & = h'(\xi_2)(A_2-A_2+1) > h'(A_2-1) = \beta (A_2-1)^{\beta-1}
	\end{align*}
	and
	\begin{align*}
		1+ (A_1+2)^\beta- A_1^\beta & = 1 + h(A_1+2) - h(A_1)                               \\
		                            & = 1 + h'(\xi_1)(A_1+2-A_1) < 3\beta(A_1+2)^{\beta-1}.
	\end{align*}
	This together with
	\begin{align*}
		\beta (A_2-1)^{\beta-1}> \beta ( 3^{\frac \beta{\beta-1}} A_1)^{\beta-1}
		=3\beta (   3A_1)^{\beta-1}> 3\beta(A_1+2)^{\beta-1}
	\end{align*}
	gives the desired inequality.

	Using again that the function $\eta(t)$
	is decreasing for $t \leq (A_1+2)\ell(I),$ in particular for $t \le (A_1-2)\ell(I)$ we have
	\begin{align*}
		\eta(t) \ge \eta\big(  (A_1-2)\ell(I)\big) & \ge \ell(I)^\beta \Big((A_1-2)^\beta +(A_2+1)^\beta \Big) \\
		                                           & > \ell(I)^\beta  (A_1^\beta+ A_2^\beta+1 ) \ge{\rm{LHS}}
		\eqref{eq:e9},
	\end{align*}
	where we used $(A_1-2)^\beta +(A_2+1)^\beta>A_1^\beta+ A_2^\beta+1$,
	Indeed, again by the mean value theorem, we have that
	\begin{align*}
		(A_2+1)^\beta- A_2^\beta = h(A_2+1)- h(A_2)> h'(A_2)(A_2+1-A_2)=\beta A_2^{\beta-1}
	\end{align*}
	and
	\begin{align*}
		1+A_1^\beta- (A_1-2)^\beta & =1+ h(A_1)-h(A_1-2)<1+ h'(A_1)(A_1-(A_1-2))                           \\
		                           & = 1+ 2\beta A_1^{\beta-1}< 3\beta A_1^{\beta-1}< \beta A_2^{\beta-1}.
	\end{align*}
	So for \eqref{eq:e9} to hold, we must have $t > (A_1 - 2)\ell(I)$.
	Notice that now by \eqref{eq:e1} we have
	$$
		|s| = x_1' - x_1 + (A_1 + A_2)\ell(I) - t,
	$$
	and so
	$$
		|s| < (1+A_1+A_2-A_1+2)\ell(I) = (A_2+3)\ell(I)
	$$
	and
	$$
		|s| > (-1+A_1+A_2 -A_1 - 2)\ell(I) = (A_2-3)\ell(I).
	$$
	We have shown the first part of the lemma.

	To see the second part of the lemma, recall that we already showed that
	\[
		\eta\big((A_1+2)\ell(I)\big) < {\rm{LHS}} \eqref{eq:e9} \qquad \textup{and} \qquad
		\eta\big((A_1-2)\ell(I)\big) > {\rm{LHS}} \eqref{eq:e9}.
	\]
	By the strict monotonicity and continuity of $\eta$, we know that there is a unique solution
	to \eqref{eq:e9} lying in $\big( (A_1-2)\ell(I), (A_1+2)\ell(I)\big)$. This shows
	the second part of the lemma.
\end{proof}

Notice that $x + \gamma(t) = y + \gamma(s)$ if and only if $x + \gamma(|s|) = y + \gamma(-t)$.
If $|s| < |-t| = t$ we can solve the latter equation using Lemma \ref{lem:11} --
we find the unique $(|s|, -t)$ so that $|s| \in ((A_1-2)\ell(I), (A_1+2)\ell(I))$ and
$t \in ((A_2-3)\ell(I), (A_2+3)\ell(I))$ with $x + \gamma(|s|) = y + \gamma(-t)$. So the original equation
$x + \gamma(t) = y + \gamma(s)$ has, for $t > |s|$, the unique solution $(t,s)$ with
$t \in ((A_2-3)\ell(I), (A_2+3)\ell(I))$ and $|s| \in ((A_1-2)\ell(I), (A_1+2)\ell(I))$. For this reason, we define
\begin{align*}
	W_2 := \big\{Q+\gamma(t) & \colon  (A_2-3)\ell(I)<t < (A_2+3)\ell(I)\big\}                               \\
	                         & \cap\big \{P + \gamma(s) \colon (A_1-2)\ell(I) < |s| < (A_1+2)\ell(I)  \big\}
\end{align*}
and notice that
\[
	\widetilde Q\cap \widetilde P= W_1\cup W_2.
\]
We now prove results related to the pair $(Q, W_1)$ -- similar results
then also hold by symmetry for the pair $(P, W_2)$.

For $x\in Q$, the set
\[
	I(x, W_1):=\big\{t \in \big( (A_1-2)\ell(I), (A_1+2)\ell(I)\big): x+\gamma(t)\in W_1 \big\}
\]
is important for us. We have the following result.
\begin{lem}\label{lem:width}
	For all $Q \in \calR_{\gamma}$ and $x\in Q$ we have
	$$
		\frac{|I(x, W_1)|}{\ell(I)}
		\sim \frac{A_1^{1-\beta} + \beta C_{\beta}^{\beta-1}}{\beta(C_{\beta}^{\beta-1}-1)}.
	$$
	In particular, we have
	\[
		\lim_{A_1\to \infty} |I(x, W_1)|
		=  \frac{C_\beta^{\beta-1}}{C_\beta^{\beta-1}-1}\ell(I) \sim \ell(I)
	\]
	and
	$$
		\lim_{A_1\to \infty} \frac{|I(x, W_1)|}{|I(x', W_1)|} = 1
	$$
	uniformly on $Q \in \calR_{\gamma}$ and $x, x' \in Q$.
\end{lem}
\begin{proof}
	Fix $x\in Q$. By definition, if $x+\gamma(t) \in W_1$, then
	there exists some $y\in P$ and $s<0$ with
	$|s|\in \big( (A_2-3)\ell(I), (A_2+3)\ell(I)\big)$ such that
	\[
		x+\gamma(t)= y+\gamma(s).
	\]
	We claim that if $t_1$ is the minimal value such that $x+\gamma(t_1)\in W_1$,
	then the corresponding $y$ should be the  left top vertex of $P;$ in other words, in the equation \eqref{eq:e9} the point $(x_1',x_2')$ should be the left top vertex of $Q.$ But this is immediate by
	\eqref{eq:e9} since a bigger $x_2'$ and a smaller $x_1'$ require a monotonically smaller $t.$

	So there is some $s_1<0$ with $-s_1 = |s_1|\in \big( (A_2-3)\ell(I), (A_2+3)\ell(I)\big)$ such that
	\begin{equation}\label{eq:e10}
		x+ \gamma(t_1)= v_{lt}+\gamma(s_1)= u_{lt}+ \gamma(A_1\ell(I))+ \gamma(A_2\ell(I)) +\gamma(s_1),
	\end{equation}
	where $v_{lt}$ and $u_{lt}$ are the left top vertices of $P$ and $Q$, respectively.
	Suppose that $h>0$ is maximal such that $x+ \gamma(t_1+h)\in W_1$. Similarly, this means that
	\begin{equation}\label{eq:e11}
		x+ \gamma(t_1+h)= u_{rb}+ \gamma(A_1\ell(I))+ \gamma(A_2\ell(I)) +\gamma(s_2),
	\end{equation}
	where $u_{rb}$ is the right bottom vertex of $Q$ and
	$-s_2\in \big( (A_2-3)\ell(I), (A_2+3)\ell(I)\big)$.
	By \eqref{eq:e10} and \eqref{eq:e11} we have
	\[
		\gamma(t_1+h)-\gamma(t_1)=u_{rb}- u_{lt}+\gamma(s_2)-\gamma(s_1).
	\]
	Componentwise this reads
	\begin{align}
		h                        & = \ell(I)-| s_2|+ |s_1|, \label{eq:e1006}                     \\
		(t_1+h)^\beta- t_1^\beta & = -\ell(I)^\beta - |s_2|^\beta +|s_1|^\beta. \label{eq:e1007}
	\end{align}
	Now, by the mean value theorem, there exist some $\xi_1$ and $\xi_2$ with
	\[
		\xi_1 \in \big((A_1-2)\ell(I), (A_1+2)\ell(I)\big)
		\qquad \textup{and} \qquad \xi_2 \in \big((A_2-3)\ell(I), (A_2+3)\ell(I)\big)
	\]
	such that
	\[
		(t_1+h)^\beta- t_1^\beta=\beta  h \xi_1^{\beta-1} \qquad \textup{and} \qquad
		|s_1|^\beta  - |s_2|^\beta = \beta (|s_1|-|s_2|) \xi_2^{\beta-1}.
	\]
	Hence, by \eqref{eq:e1007} and \eqref{eq:e1006} we have
	\begin{equation}\label{eq:e12}
		\beta h \xi_1^{\beta-1}= -\ell(I)^\beta +\beta (|s_1|-|s_2|) \xi_2^{\beta-1}
		= -\ell(I)^\beta +\beta (h-\ell(I)) \xi_2^{\beta-1},
	\end{equation}
	from which we get
	\[
		|I(x, W_1)|=h= \frac{\ell(I)^\beta+ \beta \ell(I) \xi_2^{\beta-1}}{\beta (\xi_2^{\beta-1}-\xi_1^{\beta-1})}
		= \frac{\ell(I)^\beta/{\xi_1^{\beta-1}}+ \beta \ell(I) (\xi_2/\xi_1)^{\beta-1}}{\beta \big((\xi_2/\xi_1)^{\beta-1}-1\big)}.
	\]
	Here we used that $x + \gamma(t) \in W_1$ for all $t \in [t_1, t_1+h]$ so that, indeed,
	$|I(x, W_1)|=h$ -- we will comment about this soon. First, however, we now get
	\begin{align*}
		\frac{|I(x, W_1)|}{\ell(I)}
		 & = \frac{(\ell(I)/\xi_1)^{\beta-1}+ \beta (\xi_2/\xi_1)^{\beta-1}}{\beta \big((\xi_2/\xi_1)^{\beta-1}-1\big)} \\
		 & \sim \frac{A_1^{1-\beta} + \beta C_{\beta}^{\beta-1}}{\beta(C_{\beta}^{\beta-1}-1)}.
	\end{align*}

	We now comment on the fact why $x + \gamma(t) \in W_1$ also for all $t \in (t_1, t_1+h)$.

	Fix such $t$.  We need to find $x'\in Q$ such that \eqref{eq:e9} holds.
	To simplify the notation we denote the lower left corner of $Q$ by $(u_1, u_2)$.
	If
	\[
		(A_1^\beta+A_2^\beta) \ell(I)^\beta + u_2-x_2 = t^\beta+ \big[ (A_1+A_2) \ell(I)+u_1-x_1-t \big]^\beta
	\]
	we are done, so suppose on the contrary, first, that
	$$
		(A_1^\beta+A_2^\beta) \ell(I)^\beta + u_2-x_2 < t^\beta+ \big[ (A_1+A_2) \ell(I)+u_1-x_1-t \big]^\beta.
	$$
	On the other hand, we have by monotonicity (using $t > t_1$) and the definition of $t_1$ that
	\begin{align*}
		t^\beta+ \big[ (A_1+A_2) \ell(I)+u_1-x_1-t \big]^\beta
		 & < t_1^\beta+ \big[ (A_1+A_2) \ell(I)+u_1-x_1-t_1 \big]^\beta     \\
		 & = (A_1^\beta+A_2^\beta) \ell(I)^\beta + (u_2+\ell(I)^\beta)-x_2.
	\end{align*}
	Therefore, there must exist $y_2\in (u_2,  u_2+\ell(I)^\beta)$ such that
	\[
		(A_1^\beta+A_2^\beta) \ell(I)^\beta + y_2-x_2= t^\beta+ \big[ (A_1+A_2) \ell(I)+u_1-x_1-t \big]^\beta
	\]
	showing that $x' = (u_1, y_2) \in Q$ works.
	Similarly, if
	\[
		t^\beta+ \big[ (A_1+A_2) \ell(I)+u_1-x_1-t \big]^\beta <  (A_1^\beta+A_2^\beta) \ell(I)^\beta + u_2-x_2,
	\]
	then we use again monotonicity to conclude that
	\begin{align*}
		(A_1^\beta+A_2^\beta) \ell(I)^\beta + u_2-x_2 & = (t_1+h)^\beta+ \big[ (A_1+A_2) \ell(I)+(u_1+\ell(I))-x_1-(t_1+h)\big]^\beta \\
		                                              & < t^\beta+ \big[ (A_1+A_2) \ell(I)+(u_1+\ell(I))-x_1-t \big]^\beta.
	\end{align*}
	Hence, there exists some $y_1\in (u_1, u_1+\ell(I))$ such that
	\[
		(A_1^\beta+A_2^\beta) \ell(I)^\beta + u_2-x_2=  t^\beta+ \big[ (A_1+A_2) \ell(I)+y_1-x_1-t \big]^\beta,
	\]
	showing that $x' = (y_1, u_2)\in Q$ works. This completes the proof.
\end{proof}

We continue studying the geometry of $W_1$.
Let $v=(v_1, v_2)$ be the lower left corner of $P$.
For  $1\le r <\infty$ define
\[
	P_r^{lt}:= \big(v_1, v_2+ \ell(I)^\beta\big)
	+ \Big[0, 2^{-r}\frac{\ell(I)}{\beta A_2^{\beta-1}}\Big]\times \big[- 2^{-r} \ell(I)^\beta, 0\big],
\]
\[
	P_r^{rb}:= \big(v_1+\ell(I), v_2\big)+
	\Big[-2^{-r}\frac{\ell(I)}{\beta A_2^{\beta-1}}, 0\Big]\times \big[0, 2^{-r} \ell(I)^\beta\big]
\]
and
$$
	P^{\cen} := P\setminus (P_1^{lt} \cup P_1^{rb}).
$$
We also define
\[
	\Delta P_r^{lt}
	= \Big\{(x_1, x_2)\in P_r^{lt}: x_1= v_1+ 2^{-r}\frac{\ell(I)}{\beta A_2^{\beta-1}}
	\,\, \text{or}\, \, x_2=v_2+ \ell(I)^\beta- 2^{-r} \ell(I)^\beta\Big\},
\]
and
\[
	\Delta P_r^{rb}
	= \Big\{(x_1, x_2)\in P_r^{rb}: x_1= v_1+\ell(I)- 2^{-r}\frac{\ell(I)}{\beta A_2^{\beta-1}}
	\,\, \text{or} \, \, x_2=v_2+ 2^{-r} \ell(I)^\beta\Big\}.
\]
To state our next result, we also introduce the following notation. For $z\in P$ and $y\in W_1$
we set
\[
	\phi(z, W_1):= \{z+ \gamma(s)\in W_1: -s\in \big((A_2-3)\ell(I), (A_2+3)\ell(I)\big)\},
\] and
\[
	I(y, P):=\{u \in \big((A_2-3)\ell(I), (A_2+3)\ell(I)\big): y+\gamma(u)\in P\}.
\]

\begin{lem}\label{lem:width1}
	Let $z\in P^{\cen} \subset P$ and $y \in \phi(z, W_1)$. Then, there holds
	\[
		|I(y, P)| \sim \frac{\ell(I)}{\beta A_2^{\beta-1}}.
	\]
	On the other hand, let $z\in \Delta P_r^{lt} \cup \Delta P_r^{rb}$, where $1 \le r < \infty$,
	and $y \in \phi(z, W_1)$. Then, there holds
	\[
		|I(y, P)|\sim 2^{-r}\frac{\ell(I)}{\beta A_2^{\beta-1}}.
	\]
	The implicit constants in the above estimates are independent of $y,z$ and $P$.
\end{lem}

To aid in understanding what is stated in Lemma \ref{lem:width1}, consider the following Figure \ref{fig:blowup}.
\begin{figure}[h]
	\centering
	\includegraphics[scale=1]{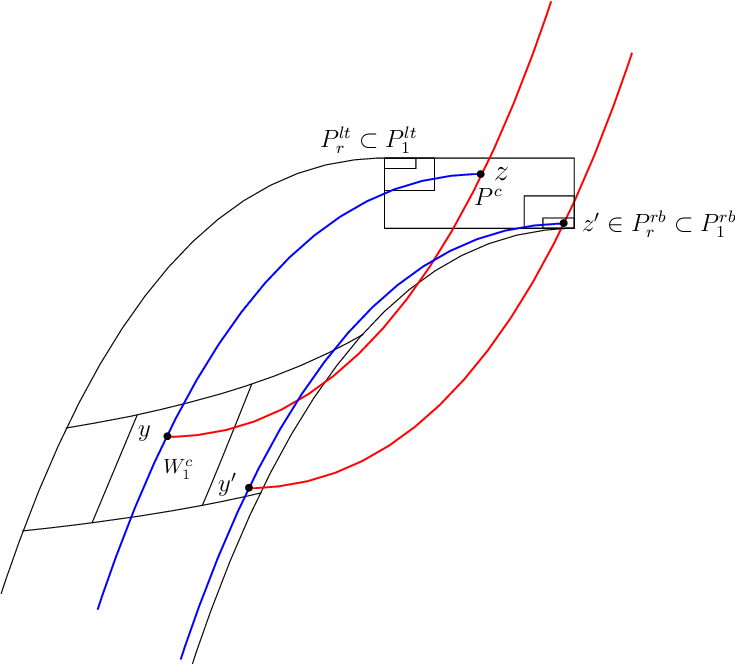}
	\caption{Explanation of Lemma \ref{lem:width1}. The function $y\mapsto |I(y,P)|$ is large when $y\in W_1^c$ but decays as $y'$ approaches the vertical part of the boundary of $W_1$.}\label{fig:blowup}
\end{figure}

\begin{proof}[Proof of Lemma \ref{lem:width1}]
	We first consider the case $z\in \Delta P_r^{lt}$, that is, we have $z\in P_r^{lt}$ and either
	\[
		z_1= v_1+ 2^{-r}\frac{\ell(I)}{\beta A_2^{\beta-1}}\qquad
		\text{or}\qquad z_2=v_2+ \ell(I)^\beta- 2^{-r} \ell(I)^\beta.
	\]
	Let then $y\in \phi(z, W_1)$. This means that
	$$
		y= z+\gamma(s)\in W_1  \qquad
		\text{for some}\qquad -s\in \big((A_2-3)\ell(I), (A_2+3)\ell(I)\big).
	$$
	Then, by definition,
	$-s \in I(y, P)$. Suppose first that $z_1= v_1+ 2^{-r}\frac{\ell(I)}{\beta A_2^{\beta-1}}$. Notice
	that we can write
	\[
		y+ \gamma (-s -h) = z+\gamma(s)+ \gamma (-s -h).
	\]
	Fix $0<h< 2^{-r}\frac{\ell(I)}{\beta(A_2+3)^{\beta-1} }$  and we will show that $y+ \gamma (-s -h)\in P.$ First, there holds that (the
	precise upper bound for $h$ is not yet used here)
	\[
		y_1+ (-s -h)- v_1=z_1-h-v_1=2^{-r} \frac{\ell(I)}{\beta A_2^{\beta-1}}-h\in (0, \ell(I))
	\]
	and
	\begin{align*}
		y_2+(-s-h)^\beta-v_2 & = z_2- v_2-|s|^\beta+(|s|-h)^\beta                           \\
		                     & \le \ell(I)^\beta -|s|^\beta+(|s|-h)^\beta\le \ell(I)^\beta.
	\end{align*}
	To obtain a lower bound for the last line, notice that by
	the mean value theorem and the upper bound for $h$ we have
	\[
		|s|^\beta-(|s|-h)^\beta<\beta |s|^{\beta-1} h<2^{-r}\ell(I)^\beta
	\]
	so that (using the location of $z_2$) we have
	\begin{align*}
		z_2- v_2-|s|^\beta+(|s|-h)^\beta & \ge (1-2^{-r})\ell(I)^\beta-(|s|^\beta-(|s|-h)^\beta) \\
		                                 & >  (1-2^{1-r})\ell(I)^\beta\ge 0.
	\end{align*}
	The above proves that $y+ \gamma (-s -h)\in P$ for all such $h$.
	Notice also that $-s-h$ must belong to the range $((A_2-3)\ell(I), (A_2+3)\ell(I))$
	-- indeed, $y + \gamma(-s-h) = w$ for some $w \in P$ and so
	$w + \gamma(s+h) = y \in W_1$ implying the claim by Lemma \ref{lem:11}.
	Thus, we have
	\[
		|I(y, P)| \ge 2^{-r}\frac{\ell(I)}{\beta(A_2+3)^{\beta-1} }.
	\]

	We discuss the corresponding upper bound next.
	Suppose $-s' \in I(y, P)$.
	Then
	\[
		y+ \gamma (-s') = z+\gamma(s)+ \gamma (-s') \in P,
	\]
	and so
	\begin{align}
		(z_1-v_1)+s-s'                 & =y_1-s'-v_1\ge 0\label{eq:eq1012},                       \\
		(z_2-v_2)-|s|^\beta+|s'|^\beta & = y_2+|s'|^\beta-v_2\le \ell(I)^\beta. \label{eq:eq1013}
	\end{align}
	By \eqref{eq:eq1012} we have
	\[
		s-s'\ge - (z_1-v_1)=- 2^{-r} \frac{\ell(I)}{\beta A_2^{\beta-1}},
	\]
	and by \eqref{eq:eq1013} we have
	\[
		-|s|^\beta+|s'|^\beta\le \ell(I)^\beta- (z_2-v_2)\le 2^{-r}\ell(I)^\beta.
	\]
	Using the mean value theorem there exists some $\xi\in ((A_2-3)\ell(I), (A_2+3)\ell(I))$ such that
	$$
		|s'|^{\beta} - |s|^{\beta} = \beta \xi^{\beta-1}(|s'| - |s|)
	$$
	and so
	\[
		s-s'=|s'|-|s|=\frac{-|s|^\beta+|s'|^\beta}{\beta \xi^{\beta-1}}<2^{-r} \frac{\ell(I)}{\beta (A_2-3)^{\beta-1}}.
	\]
	Therefore,
	\[
		|I(y, P)| \le \frac{2^{-r+1}}{\beta(A_2-3)^{\beta-1} }\ell(I).
	\]
	This completes the proof of the case $z\in \Delta P_r^{lt}$ with
	$z_1= v_1+ 2^{-r}\frac{\ell(I)}{\beta A_2^{\beta-1}}$.

	The case $z_2=v_2+ \ell(I)^\beta- 2^{-r} \ell(I)^\beta$ is similar.
	The argument for $z\in P^{\cen}$ is essentially the same, too
	(if $z_2-v_2> \ell(I)^\beta/2$ this will be similar as the proof presented above,
	and if $z_2-v_2\le \ell(I)^\beta/2$ it will be similar as the case
	$z\in \Delta P_r^{lt}$ with $z_2=v_2+ \ell(I)^\beta- 2^{-r} \ell(I)^\beta$).
	Finally, the case $\Delta P_r^{rb}$ follows with similar arguments as well.
\end{proof}

\subsection*{Auxiliary functions}
We next introduce certain auxiliary functions
that will be important in the upcoming weak factorization argument. Define
\[
	g_P:=1_P,\qquad g_Q:=1_Q.
\]
We will next define $g_{W_1}$, which is more complicated. Let
\begin{align*}
	W_1^{\cen} :=\{y\in W_1: \exists z\in P^{\cen} \text{\,\,such that\,\,} y\in \phi(z, W_1)\}
\end{align*}
and let $\eta\ge 0$ be the smallest constant so that
\[
	W_1 = \bigcup\limits_{r\in [-\eta, \infty)}W_1(r), \quad  W_1(r)
	:= \Big\{y\in W_1: |I(y, P)|=2^{-r} \frac{\ell(I)}{\beta A_2^{\beta-1}}\Big\}.
\]
Note that the existence of such an $\eta$ is guaranteed by Lemma \ref{lem:width1}. Then we define
\[
	\varphi(y, M) := \sum_{-\eta\le r<M}1_{W_1(r)}(y)
	+ \sum_{ r\ge M}1_{W_1(r)}(y) 2^{-(r-M)},\qquad y \in W_1,\quad M>0,
\]
where $M$ is sufficiently large to be fixed during the proof of the next Lemma \ref{lem:gw1}.
\begin{lem}\label{lem:gw1}
	There exists an absolute constant $M>1$ such that
	$$
		g_{W_1}(y) := \varphi(y, M)
	$$
	satisfies the following properties:
	\begin{align}
		 & 1_{W_1^{\cen}}g_{W_1}= 1_{W_1^{\cen}}\label{eq:e18},                                                      \\
		 & \qquad g_{W_1}(y)\sim_M |I(y,P)|\frac {\beta A_2^{\beta-1}}{\ell(I)} \label{eq:e19},                      \\
		 & \qquad\qquad\lim_{A_1\to \infty} \frac{g_{W_1}(z+\gamma (s))}{g_{W_1}(z+\gamma (s'))} =1 \label{eq:e20z},
	\end{align}
	where in \eqref{eq:e19} the implicit constants do not depend on $y \in W_1$ and $Q$,
	and in \eqref{eq:e20z} the limit is uniform on $Q\in \mathcal R_\gamma$,
	$z \in P\setminus  \{v_{lt}, v_{rb}\}$ and
	\begin{align*}
		s, s' & \in I(z, W_1)                                                                                 \\
		      & := \{s \colon -s \in ((A_2-3)\ell(I), (A_2+3)\ell(I)) \textup{ and } z + \gamma(s) \in W_1\}.
	\end{align*}
\end{lem}
\begin{proof}
	We first prove \eqref{eq:e18}. For $y\in W_1^{\cen}$ by Lemma \ref{lem:width1} we have
	\[
		|I(y, P)|\sim \frac{\ell(I)}{\beta A_2^{\beta-1}},
	\]
	and therefore if $M$ is such that
	\[
		|I(y, P)| > 2^{-M} \frac{\ell(I)}{\beta A_2^{\beta-1}},
	\]
	then $y\not\in W_1(r)$ for $r\geq M,$ hence $g_{W_1}=1$ on $W_1^{\cen}$ and \eqref{eq:e18} is proved.
	The claim \eqref{eq:e19} is immediate from definition.

	It remains to verify	 \eqref{eq:e20z}. Fix for the remaining argument a point
	$z\in P\setminus \{v_{lt}, v_{rb}\}$.
	Let
	$-s,u\in \big((A_2-3)\ell(I), (A_2+3)\ell(I)\big)$ satisfy $z+\gamma(s)+\gamma(u)\in P$, i.e. with $(v_1, v_2)$ being the lower left corner of $P$ there holds that
	\begin{align}\label{eq:noExit}
		0\le z_1+s+u-v_1\le \ell(I),\quad 0\le z_2-|s|^\beta+u^\beta -v_2\le \ell(I)^\beta.
	\end{align}
	Now let $a(s)$ stand for the minimal $u_*$ such that both of the above bounds \eqref{eq:noExit} hold, and  let $b(s)$ stand for the maximal $u^*$ such that both hold; then it is clear that
	\begin{align}\label{eq:cont}
		|I(z+\gamma(s), P)|=b(s)-a(s) = u^*-u_*.
	\end{align}
	Explicitly, it is clear from \eqref{eq:noExit} that
	\begin{align*}
		a(s) & := \max\{ v_1-z_1-s, \,(v_2-z_2+|s|^\beta)^{\frac 1\beta}\},                       \\
		b(s) & := \min \{ v_1+\ell(I)-z_1-s, (v_2+\ell(I)^\beta -z_2+|s|^\beta)^{\frac 1\beta}\}.
	\end{align*}
	Now $|I(z+\gamma(s), P)|$ is positive (by $z\in P\setminus \{v_{lt}, v_{rb}\}$) and from \eqref{eq:cont} it is clearly continuous, and then also
	\[
		r(s):=-\log_2 \Big(|I(z+\gamma(s), P)|\frac{\beta A_2^{\beta-1}}{\ell(I) }\Big)
	\]
	is continuous.
	By definition $z+\gamma(s)\in W_1(r(s)),$ i.e. if $r(s) \le M,$ then
	$$
		g_{W_1}(z+\gamma(s)) = 1
	$$
	and if $r(s) \ge M,$ then
	$$
		g_{W_1}(z+\gamma(s)) = 2^{-(r(s)-M)} = 2^M \frac{\beta A_2^{\beta-1}}{\ell(I)} |I(z+\gamma(s), P)|.
	$$
	So if $r(s), r(s') \le M$, then
	$$
		\frac{g_{W_1}(z+\gamma (s))}{g_{W_1}(z+\gamma (s'))} = \frac{1}{1} = 1,
	$$
	and if $r(s), r(s') \ge M$, then
	$$
		\frac{g_{W_1}(z+\gamma (s))}{g_{W_1}(z+\gamma (s'))} = \frac{|I(z+\gamma(s), P)|}{|I(z+\gamma(s'), P)|}.
	$$
	Suppose that $r(s) \ge M$ and $r(s') \le M$. By the continuity of $r(s)$ we find $s_0$ between $s$ and $s'$
	such that $r(s_0) = M$. We then have
	$$
		\frac{g_{W_1}(z+\gamma (s))}{g_{W_1}(z+\gamma (s'))}
		= \frac{g_{W_1}(z+\gamma (s))}{1} = \frac{g_{W_1}(z+\gamma (s))}{g_{W_1}(z+\gamma (s_0))}
		= \frac{|I(z+\gamma(s), P)|}{|I(z+\gamma(s_0), P)|}.
	$$
	Similarly, if $r(s) \le M$ and $r(s') \ge M$, we again find such an $s_0,$ and again
	$$
		\frac{g_{W_1}(z+\gamma (s))}{g_{W_1}(z+\gamma (s'))}
		= \frac{1}{g_{W_1}(z+\gamma (s'))} = \frac{g_{W_1}(z+\gamma(s_0))}{g_{W_1}(z+\gamma (s'))}
		= \frac{|I(z+\gamma(s_0), P)|}{|I(z+\gamma(s'), P)|}.
	$$
	So if the limit \eqref{eq:e20z} exists and is uniform, then
	\[
		\lim_{A_1\to \infty} \frac{g_{W_1}(z+\gamma (s))}{g_{W_1}(z+\gamma (s'))} =   \lim_{A_1\to \infty} \frac{|I(z+\gamma (s), P)|}{|I(z+\gamma (s'), P)|}.
	\]
	In the following Lemma \ref{lem:inP} we show that this is indeed so.
\end{proof}
\begin{lem}\label{lem:inP} There holds that
	\begin{align}\label{eq:e20x}
		\lim_{A_1\to \infty} \frac{|I(z+\gamma (s'), P)|}{|I(z+\gamma (s), P)|}=1.
	\end{align}
	In particular, to be recorded for later use in the proof of Proposition \ref{prop:awf}, there holds that
	\begin{align}\label{eq:e20}
		\lim_{A_1\to \infty} \frac{g_{W_1}(z+\gamma (s))}{g_{W_1}(z+\gamma (s'))}\cdot \frac{|I(z+\gamma (s'), P)|}{|I(z+\gamma (s), P)|}=1.
	\end{align}
	Crucially, in \eqref{eq:e20x} and \eqref{eq:e20} the implicit constants do not depend on $Q$
	and the limits are uniform on $Q\in \mathcal R_\gamma$,
	$z \in P\setminus  \{v_{lt}, v_{rb}\}$ and
	$
		s, s' \in I(z, W_1).
	$
\end{lem}

\begin{proof}
	Recall that $|I(z+\gamma(s), P)|=b(s)-a(s),$ where
	\begin{align*}
		a(s) & =\max\{ v_1-z_1-s, \,(v_2-z_2+|s|^\beta)^{\frac 1\beta}\},                        \\
		b(s) & = \min \{ v_1+\ell(I)-z_1-s, (v_2+\ell(I)^\beta -z_2+|s|^\beta)^{\frac 1\beta}\}.
	\end{align*}
	We have four cases to consider.

		{\bf Case 1}. $b(s')= (v_2+\ell(I)^\beta -z_2+|s'|^\beta)^{\frac 1\beta}$, $a(s')=v_1-z_1-s'$. In this case we have
	\begin{align*}
		b(s')-a(s') & = (v_2+\ell(I)^\beta -z_2+|s'|^\beta)^{\frac 1\beta}- (v_1-z_1-s')                                                \\
		            & = (z_1-v_1)+  (v_2+\ell(I)^\beta -z_2+|s'|^\beta)^{\frac 1\beta}- (|s'|^\beta)^{\frac1\beta}                      \\
		            & \ge  (z_1-v_1)+ \frac {v_2+\ell(I)^\beta -z_2}{\beta (v_2+\ell(I)^\beta -z_2+|s'|^\beta)^{1-\frac 1\beta}}        \\
		            & \ge  (z_1-v_1)+ \frac {v_2+\ell(I)^\beta -z_2}{\beta\big((A_2+3)^\beta+1\big)^{1-\frac 1\beta}\ell(I)^{\beta-1}}.
	\end{align*}
	Similarly, we have
	\begin{align*}
		b(s)-a(s) & \le (v_2+\ell(I)^\beta -z_2+|s|^\beta)^{\frac 1\beta}- (v_1-z_1-s)                      \\
		          & \le (z_1-v_1)+ \frac {v_2+\ell(I)^\beta -z_2}{\beta\big((A_2-3)\ell(I)\big)^{\beta-1}}.
	\end{align*}
	We record the simple fact that for $0< x< y$ and $a, b\ge 0$ with $\max\{a,b\}>0$ we have
	\begin{equation}\label{eq:e1025}
		\frac{a+bx}{a+by}\ge \frac{axy^{-1}+bx}{a+by}
		= \frac{x(ay^{-1}+b)}{y(ay^{-1}+b)}
		= \frac xy.
	\end{equation}
	Using this we get
	\begin{align*}
		\frac{|I(z+\gamma(s'), P)|}{|I(z+\gamma(s), P)|}
		\ge \frac{(A_2-3)^{\beta-1}}{\big((A_2+3)^\beta+1\big)^{1-\frac 1\beta}}.
	\end{align*}
	We are content with this for the time being (we calculate the limit after all of the cases), and move on now.

		{\bf Case 2}. $b(s')= (v_2+\ell(I)^\beta -z_2+|s'|^\beta)^{\frac 1\beta}$, $a(s')=(v_2-z_2+|s'|^\beta)^{\frac 1\beta}$. In this case we have
	\begin{align*}
		b(s')-a(s') & = (v_2+\ell(I)^\beta -z_2+|s'|^\beta)^{\frac 1\beta}-(v_2-z_2+|s'|^\beta)^{\frac 1\beta}                                                                \\
		            & \ge  \frac {\ell(I)^\beta}{\beta (v_2+\ell(I)^\beta  -z_2+|s'|^\beta)^{1-\frac 1\beta} }\ge \frac {\ell(I)}{\beta ((A_2+3)^\beta+1)^{1-\frac 1\beta} }.
	\end{align*}
	Similarly, we have
	\begin{align*}
		b(s)-a(s) & \le (v_2+\ell(I)^\beta -z_2+|s|^\beta)^{\frac 1\beta}-(v_2-z_2+|s|^\beta)^{\frac 1\beta} \\
		          & \le \frac {\ell(I)}{\beta ((A_2-3)^\beta-1)^{1-\frac 1\beta} }.
	\end{align*}
	Hence, in this case we obtain
	\begin{align*}
		\frac{|I(z+\gamma(s'), P)|}{|I(z+\gamma(s), P)|}\ge \frac{((A_2-3)^\beta-1)^{1-\frac 1\beta}}{\big((A_2+3)^\beta+1\big)^{1-\frac 1\beta}}.
	\end{align*}

	{\bf Case 3}. $b(s')=v_1+\ell(I)-z_1-s'$, $a(s')=(v_2-z_2+|s'|^\beta)^{\frac 1\beta}$. Notice that now
	\begin{align*}
		b(s')-a(s') & = v_1+\ell(I)-z_1-s'-(v_2-z_2+|s'|^\beta)^{\frac 1\beta}                                   \\
		            & = v_1+\ell(I)-z_1+\big( |s'|^\beta\big)^{\frac 1\beta}-(v_2-z_2+|s'|^\beta)^{\frac 1\beta} \\
		            & \ge v_1+\ell(I)-z_1+\frac{z_2-v_2}{\beta \big( |s'|^{\beta}\big)^{1-\frac 1\beta} }        \\
		            & \ge v_1+\ell(I)-z_1+\frac{z_2-v_2}{\beta \big( (A_2+3)\ell(I)\big)^{\beta-1} }.
	\end{align*}
	Likewise, we have
	\begin{align*}
		b(s)-a(s) & \le  v_1+\ell(I)-z_1+\big( |s|^\beta\big)^{\frac 1\beta}-(v_2-z_2+|s|^\beta)^{\frac 1\beta}              \\
		          & \le v_1+\ell(I)-z_1+\frac{z_2-v_2}{\beta \big( (A_2-3)^\beta-1\big)^{1-\frac 1\beta}\ell(I)^{\beta-1} }.
	\end{align*}
	Then, by using \eqref{eq:e1025} we get
	\begin{align*}
		\frac{|I(z+\gamma(s'), P)|}{|I(z+\gamma(s), P)|}\ge \frac{((A_2-3)^\beta-1)^{1-\frac 1\beta}}{ (A_2+3)^{\beta-1}}.
	\end{align*}

	{\bf Case 4}. $b(s')=v_1+\ell(I)-z_1-s'$, $a(s')=v_1-z_1-s'$. Now trivially
	$b(s')-a(s')=\ell(I)$ and also
	\[
		b(s)-a(s)\le v_1+\ell(I)-z_1-s-(v_1-z_1-s)=\ell(I).
	\]
	Thus, we get
	\[
		\frac{|I(z+\gamma(s'), P)|}{|I(z+\gamma(s), P)|}\ge 1.
	\]

	In summary, for all cases we have
	\begin{align*}
		\frac{|I(z+\gamma(s'), P)|}{|I(z+\gamma(s), P)|}
		\ge \frac{((A_2-3)^\beta-1)^{1-\frac 1\beta}}{ \big((A_2+3)^{\beta}+1\big)^{1-\frac 1\beta}}.
	\end{align*}
	The assumptions on $s, s'$ were symmetric, so switching the roles of $s$ and $s'$ gives that we also have
	\[
		\frac{|I(z+\gamma(s), P)|}{|I(z+\gamma(s'), P)|}
		\ge \frac{((A_2-3)^\beta-1)^{1-\frac 1\beta}}{ \big((A_2+3)^{\beta}+1\big)^{1-\frac 1\beta}}.
	\]
	Therefore,
	\begin{align*}
		\frac{((A_2+3)^\beta+1)^{1-\frac 1\beta}}{ \big((A_2-3)^{\beta}-1\big)^{1-\frac 1\beta}}	 \ge \frac{|I(z+\gamma(s'), P)|}{|I(z+\gamma(s), P)|}
		\ge \frac{((A_2-3)^\beta-1)^{1-\frac 1\beta}}{ \big((A_2+3)^{\beta}+1\big)^{1-\frac 1\beta}},
	\end{align*}
	and we conclude \eqref{eq:e20x}.
\end{proof}

\section{Approximate weak factorization}\label{sect:AWF}
We begin with the following proposition, which is the main step of our approximate weak factorization argument.
\begin{prop}\label{prop:awf}
	Let $f\in L^{\infty}$ be supported on $Q=I\times J\in \mathcal R_\gamma $. Then for all $A_1$ large enough
	(independently of $Q$), the function $f$ can be written as
	\[
		f= [h_Q H_\gamma^* g_{W_1}- g_{W_1}H_\gamma h_Q]
		+ [h_{W_1}H_\gamma g_P- g_P H_\gamma^* h_{W_1}]+ \widetilde f_P,
	\]
	where
	\[
		h_Q := \frac f{H_\gamma^* g_{W_1}},\quad h_{W_1} := \frac{g_{W_1} H_\gamma h_Q}{ H_\gamma g_P},\quad
		\widetilde f_P := g_P H_\gamma^* \Bigg( \frac{g_{W_1}}{H_\gamma g_P}
		H_\gamma \Big(\frac f{H_\gamma^* g_{W_1}}\Big)\Bigg)
	\]
	and there holds that
	\[
		|h_Q|\lesssim A_1 |f|,\qquad |h_{W_1}|\lesssim A_1 \|f\|_\infty 1_{W_1}.
	\]
	Moreover, suppose that $\int_Q f=0$ and let  $\varepsilon>0$. Then for all $A_1$ large enough (independently of $Q$), there holds that
	\begin{align}\label{eq:prop:AWF1}
		\int_{P} \widetilde f_P =0,\qquad |\widetilde f_P| \lesssim \varepsilon \|f\|_\infty 1_P.
	\end{align}
\end{prop}

\begin{proof}
	The identity for $f$ (provided that all of the functions are well-defined) is simple
	(everything just cancels out leaving $f$).
	So, it only remains to deal with the estimates.

	Fix $x\in Q$. We denote
	\[
		I^{\cen}(x, W_1)= \big\{ t\in I(x, W_1): g_{W_1}( x+\gamma(t))=1\big\}.
	\]
	By Lemma \ref{lem:gw1} we have that, if $x+\gamma(t)= y+ \gamma(s)\in W_1$ with $y\in P^{\cen}$, then
	\[
		g_{W_1}( x+\gamma(t))=1.
	\]
	Therefore, we have
	\begin{equation}\label{eq:e40}
		I(x, W_1)\setminus I^{\cen}(x, W_1)\subset  \big\{ t: \exists y\in P\setminus P^{\cen}
		\,\, \text{such that}\,\,  x+\gamma(t)= y+ \gamma(s)\in W_1\big\}.
	\end{equation}
	To estimate $|I(x, W_1)\setminus I^{\cen}(x, W_1) |$ we will follow ideas from the proof
	of Lemma \ref{lem:width}. Indeed, in the proof of Lemma \ref{lem:width} we had a
	$t_1$ associated with the top left corner of $P$, and a $t_1+h$ associated with the bottom right
	corner of $P$. Here we will have some $h_1$ associated with $P_1^{lt}$ and some $h_2$
	associated with $P_1^{rb}$. Indeed, using the notation from the proof of Lemma \ref{lem:width}
	we have the following analogies of \eqref{eq:e1006} and \eqref{eq:e1007}:
	\begin{align*}
		h_1                        & = \frac{\ell(I)}{2\beta A_2^{\beta-1}}-| s_2|+ |s_1|, \\
		(t_1+h_1)^\beta- t_1^\beta & = -\frac{\ell(I)^\beta}2 - |s_2|^\beta +|s_1|^\beta.
	\end{align*}
	Following the computations there we get
	\begin{align*}
		h_1\le \frac{\frac{\ell(I)}{2(A_1-2)^{\beta-1}} + \frac{\beta\ell(I)}{2\beta A_2^{\beta-1}}\cdot (\frac{A_2+3}{A_1-2})^{\beta-1}}{\beta [(\frac{A_2-3}{A_1+2})^{\beta-1}-1]}\lesssim \frac{\ell(I)}{\beta A_2^{\beta-1}}.
	\end{align*}
	The estimate for $h_2$ is the same. Hence
	\[
		|I(x, W_1)\setminus I^{\cen}(x, W_1)|\le
		h_1+h_2\lesssim \frac{\ell(I)}{\beta A_2^{\beta-1}}\le \frac{|I(x, W_1)|}{\beta A_2^{\beta-1}}.
	\]

	Notice now that
	$$
		H_\gamma^* g_{W_1}(x) = \int_{I(x, W_1)} g_{W_1} (x+\gamma (t))  \frac{\ud t}{t}
		\ge \int_{I^{\cen}(x, W_1)} \frac{\ud t}{t} \ge \frac{|I^{\cen}(x, W_1)|}{(A_1+2)\ell(I)}.
	$$
	Hence, we get
	\begin{align}\label{eq:55}
		\Big(1-\frac c{\beta A_2^{\beta-1}}\Big)\frac{|I(x, W_1)|}{(A_1+2)\ell(I)}
		 & \le \frac{|I^{\cen}(x, W_1)|}{(A_1+2)\ell(I)}
		\le H_\gamma^* g_{W_1}(x)                                                                                  \\
		 & = \int_{I(x, W_1)} g_{W_1} (x+\gamma (t)) \frac{\ud t}t\le \frac{|I(x, W_1)|}{(A_1-2)\ell(I)}.\nonumber
	\end{align}
	In particular, by Lemma \ref{lem:width} we have
	\[
		|h_Q(x)| \lesssim A_1 |f(x)|.
	\]
	provided that $A_1$ (and hence $A_2$) is sufficiently large.

	Next, we estimate $h_{W_1}(y)$ for $y \in W_1$. To this end, we first estimate
	$H_\gamma g_P(y)$ (recall $g_P = 1_P$). By a direct computation, we have
	\begin{align*}
		H_\gamma g_P(y)= \int_{\mathbb R}g_P(y-\gamma(t))\frac{\ud t}{t}
		 & =- \int_{\mathbb R}g_P(y+\gamma(t))\frac{\ud t}{t}                      \\
		 & = - \int_{I(y, P)}\frac{\ud t}{t}\sim - \frac{| I(y, P) |}{A_2\ell(I)}.
	\end{align*}
	Then, for $H_{\gamma} h_Q$ we have
	\begin{align*}
		\big| H_\gamma h_Q(y)\big|= \Big|\int_{\mathbb R}h_Q(y-\gamma(t))\frac{\ud t}{t}\Big|
		 & =\Big| \int_{I(y, Q)} h_Q(y-\gamma(t))\frac{\ud t}{t} \Big|\lesssim \|h_Q\|_\infty   \frac{| I(y, Q) |}{A_1\ell(I)},
	\end{align*}
	where
	\[
		I(y, Q) := \{t: y-\gamma(t)\in Q\}.
	\]
	Observe that if $t_1, t_2\in  I(y, Q)$, then
	\[
		\ell(I)^\beta \ge |(y_2- t_1^\beta)-  (y_2- t_2^\beta)|\ge \beta [(A_1-2)\ell(I)]^{\beta-1}|t_2-t_1|.
	\]
	Hence, we have
	\[
		| I(y, Q) |\lesssim \frac{\ell(I)}{\beta A_1^{\beta-1}}.
	\]
	Now, by \eqref{eq:e19}, we have
	\begin{align*}
		|h_{W_1}(y)| & =\Big|\frac{g_{W_1}(y) H_\gamma h_Q(y)}{ H_\gamma g_P(y)}\Big|                      \\
		             & \lesssim \|h_Q\|_\infty  \frac{g_{W_1}(y)A_2 | I(y, Q) |}{A_1| I(y, P) |}           \\
		             & \lesssim A_1 \|f\|_{\infty} \cdot |I(y,P)| \frac{\beta A_2^{\beta-1}}{\ell(I)}
		\cdot \frac{A_2}{A_1} \cdot \frac{\ell(I)}{\beta A_1^{\beta-1}} \cdot \frac{1}{|I(y,P)|}1_{W_1}(y) \\
		             & = A_1 \|f\|_\infty \frac{A_2^{\beta}}{A_1^{\beta}}1_{W_1}(y)
		\lesssim A_1 \|f\|_\infty 1_{W_1}(y).
	\end{align*}

	Having now estimated $h_Q$ and $h_{W_1}$,
	it only remains to deal with $\widetilde f_P$.
	So we are now assuming $\int_Q f = 0$.
	The property $\int_P \widetilde f_P=0$ is obvious:
	\begin{align*}
		\int g_P H_\gamma^* \Bigg( \frac{g_{W_1}}{H_\gamma g_P}
		H_\gamma \Big(\frac f{H_\gamma^* g_{W_1}}\Big)\Bigg)
		= \int g_{W_1} H_\gamma \Big(\frac f{H_\gamma^* g_{W_1}}\Big)
		= \int f = 0.
	\end{align*}
	Fix now $z\in P$. We can write
	\begin{align*}
		\widetilde f_P(z) & = H_\gamma^* \Bigg( \frac{g_{W_1}}{H_\gamma g_P}
		H_\gamma \Big(\frac f{H_\gamma^* g_{W_1}}\Big)\Bigg)(z)                                     \\
		                  & = \int_{I(z, W_1)}  \Big(\frac{g_{W_1}}{H_\gamma g_P}\Big)(z+\gamma(s))
		H_\gamma \Big(\frac f{H_\gamma^* g_{W_1}}\Big) (z+\gamma(s)) \frac{\ud s}{s}                \\
		                  & =  \int_{I(z, W_1)} \int_{I(z+\gamma(s), Q)}
		\Big(\frac{g_{W_1}}{H_\gamma g_P}\Big)(z+\gamma(s))
		\frac {f( z+\gamma(s)-\gamma(t))}{(H_\gamma^* g_{W_1}) (z+\gamma(s)-\gamma(t) )}
		\frac{\ud t}{t}\frac{\ud s}{s}                                                              \\
		                  & =: \int_{I(z, W_1)} \int_{I(z+\gamma(s), Q)} C_{z}(t,s)
		f( z+\gamma(s)-\gamma(t)) \frac{\ud t}{t}\frac{\ud s}{s},
	\end{align*}
	where
	\[
		C_{z}(t,s):= \Big(\frac{g_{W_1}}{H_\gamma g_P}\Big)(z+\gamma(s))
		\frac {1}{(H_\gamma^* g_{W_1}) (z+\gamma(s)-\gamma(t) )}.
	\]
	Observe that if $z\in \{v_{lt}, v_{rb}\}$ then $g_{W_1}(z+\gamma(s))=0$ and trivially $\widetilde f_P(z)=0$.
	So we may assume $z\in P\setminus \{v_{lt}, v_{rb}\}$.
	By \eqref{eq:55}, we have that
	\[
		\frac{(A_1-2)\ell(I)}{ |I(z+\gamma(s)-\gamma(t), W_1)|}\le \frac {1}{(H_\gamma^* g_{W_1}) (z+\gamma(s)-\gamma(t) )}
		\le \frac{(A_1+2)\ell(I)}{ |I^{\cen}(z+\gamma(s)-\gamma(t), W_1)|}.
	\]
	As above, by
	\[
		H_\gamma g_P(z+\gamma(s))= - \int_{I(z+\gamma(s), P)}\frac{\ud t}{t}
	\]
	we have
	\begin{equation}\label{eq:gp}
		- \frac{|I(z+\gamma(s), P) |}{ (A_2-3)\ell(I)}
		\le H_\gamma g_P(z+\gamma(s))
		\le - \frac{|I(z+\gamma(s), P) |}{ (A_2+3)\ell(I)}.
	\end{equation}
	Now, define the set
	\[
		F_z:=\{(t,s): z+\gamma(s)\in W_1, \quad z+\gamma(s)-\gamma(t)\in Q\}
	\]
	and fix $(t_0, s_0)\in F_z$. By definition,
	\begin{align*}
		\frac{C_{z}(t,s)}{C_{z}(t_0,s_0)} & = \frac{g_{W_1}(z+\gamma(s))}{g_{W_1}(z+\gamma(s_0))}
		\cdot \frac{H_\gamma g_P(z+\gamma(s_0))}{H_\gamma g_P(z+\gamma(s))}
		\cdot \frac {(H_\gamma^* g_{W_1}) (z+\gamma(s_0)-\gamma(t_0) )}
		{(H_\gamma^* g_{W_1}) (z+\gamma(s)-\gamma(t) )}                                              \\
		                                  & \sim \frac{g_{W_1}(z+\gamma(s))}{g_{W_1}(z+\gamma(s_0))}
		\cdot \frac{|I(z+\gamma(s_0), P) |}{|I(z+\gamma(s), P) |}\cdot
		\frac{ |I(z+\gamma(s_0)-\gamma(t_0), W_1)|}
		{ |I(z+\gamma(s)-\gamma(t), W_1)|}.
	\end{align*}
	In particular, by \eqref{eq:e20} and Lemma \ref{lem:width} we have
	\[
		\lim_{A_1\to \infty}\frac{C_{z}(t,s)}{C_{z}(t_0,s_0)}=1
	\]
	uniformly on $z, t, s, t_0, s_0$ and $Q$. In other words,
	\[
		\lim_{A_1\to \infty} \sup_{\substack{z\in P\setminus\{v_{lt}, v_{rb}\}\\(t,s)\in F_z}}
		\frac{|C_{z}(t,s)-C_{z}(t_0,s_0)|}{|C_{z}(t_0,s_0)|}=0.
	\]
	However, by \eqref{eq:gp}, \eqref{eq:e19}, \eqref{eq:55} and Lemma \ref{lem:width} we have
	\begin{align}\label{eq:110}
		|C_{z}(t_0,s_0)| & \le \frac{(A_1+2)\ell(I) }{|I^{\cen}(z+\gamma(s_0)-\gamma(t_0), W_1)|}
		\frac{g_{W_1}(z+\gamma(s_0))(A_2+3)\ell(I) }{|I(z+\gamma(s_0), P) | }  \lesssim  A_1^{\beta+1}.
	\end{align}
	So we get
	\[
		\lim_{A_1\to \infty}   A_1^{-(\beta+1)}\sup
		\limits_{\substack{z\in P\setminus\{v_{lt}, v_{rb}\}\\(t,s)\in F_z}}
		|C_{z}(t,s)-C_{z}(t_0,s_0) |  =0.
	\]
	This means that for any $\varepsilon>0$, there exists some $A(\varepsilon)>0$
	such that as long as $A_1>A(\varepsilon)$, we have
	\[
		\sup\limits_{\substack{z\in P\setminus\{v_{lt}, v_{rb}\}\\(t,s)\in F_z}}
		|C_{z}(t,s)-C_{z}(t_0,s_0) | <\varepsilon A_1^{\beta+1}.
	\]
	Hence, if we define
	\[
		D_{z}f :=\int_{I(z, W_1)}
		\int_{I(z+\gamma(s), Q)} (C_{z}(t,s)- C_{z}(t_0,s_0) ) f( z+\gamma(s)-\gamma(t))
		\frac{\ud t}{t}\frac{\ud s}{s},
	\]
	we have
	\begin{align*}
		|D_{z}f| & \le \varepsilon A_1^{\beta+1}  \|f\|_\infty  \int_{I(z, W_1)}
		\int_{I(z+\gamma(s), Q)}\frac{\ud t}{|t|}\frac{\ud s}{|s|}                  \\
		         & \lesssim \varepsilon A_1^{\beta+1} \|f\|_\infty \int_{I(z, W_1)}
		\frac{| I(z+\gamma(s), Q)|}{A_1  \ell(I)}\frac{\ud s}{|s|}                  \\
		         & \lesssim  \varepsilon A_1^{\beta+1} \|f\|_\infty
		\frac{\ell(I)}{\beta A_1^{\beta-1}} \frac{|I(z, W_1)|}{A_1 A_2 \ell(I)^2}
		\lesssim \varepsilon \|f\|_\infty,
	\end{align*}
	where we have used the already familiar bound
	\[
		| I(z+\gamma(s), Q)|\lesssim \frac{\ell(I)}{\beta A_1^{\beta-1}}
	\]
	and
	\[
		|I(z, W_1) |\le  |  (-(A_2+3)\ell(I), -(A_2-3)\ell(I))| =6\ell(I).
	\]

	It now remains to control
	\[
		\widetilde f_P(z)- D_{z}f
		= C_{z}(t_0,s_0)  \int_{I(z, W_1)} \int_{I(z+\gamma(s), Q)}
		f( z+\gamma(s)-\gamma(t)) \frac{\ud t}{t}\frac{\ud s}{s}.
	\]
	Consider the map $h_z\colon F_z \to Q$ defined by
	\[
		h_z(t, s)=z+\gamma(s)-\gamma(t).
	\]
	For any $x\in Q$ and $z\in P\setminus \{v_{lt}, v_{rb}\}$, by Lemma \ref{lem:11} there is
	a unique pair $(t,s)$ with
	\[
		x+\gamma(t) = z+\gamma(s)\in W_1.
	\]
	This means that given any $x\in Q$ and $z\in P$, we
	find a unique $(t_{x,z},s_{x,z})$ such that $(t_{x,z},s_{x,z})\in F_z$. Thus, $h_z$ is bijective,
	and clearly $h_z$ is differentiable. Set $\rho(t,s) := (ts)^{-1}$.
	By a change of variables, we have
	\begin{align*}
		\int_{I(z, W_1)} \int_{I(z+\gamma(s), Q)}  f( z+\gamma(s)-\gamma(t)) \frac{\ud t}{t}\frac{\ud s}{s}
		 & = \int_{F_z} f(h_z(t,s))\rho(t,s)\ud (t,s)                                      \\
		 & = \int_{Q} f(x) \rho (h_z^{-1} (x)) \det J_{h_z^{-1}}(x) \ud x                  \\
		 & = \int_{Q} f(x) \rho (h_z^{-1} (x))\frac{\ud x}{\det J_{h_z}(t_{x,z},s_{x,z})},
	\end{align*}
	where we denote $(t_{x,z},s_{x,z}) = h_z^{-1}(x).$
	A direct computation gives that
	\[
		\det J_{h_z}(t_{x,z},s_{x,z})= \det \begin{bmatrix}
			-1                       & 1                          \\
			-\beta t_{x,z}^{\beta-1} & \beta (-s_{x,z})^{\beta-1}
		\end{bmatrix} = \beta (t_{x,z}^{\beta-1}- (-s_{x,z})^{\beta-1}).
	\]
	Due to this we define
	\[
		\theta_z(x)= \frac{1}{\beta (t_{x,z}^{\beta-1}- (-s_{x,z})^{\beta-1}) } \frac 1{ t_{x,z}s_{x,z}}
		, \quad \textup{where} \quad t_{x,z}\sim A_1 \ell(I),\quad s_{x,z}\sim - A_2 \ell(I).
	\]
	We have
	\begin{align*}
		\int_{I(z, W_1)} & \int_{I(z+\gamma(s), Q)}  f( z+\gamma(s)-\gamma(t)) \frac{\ud t}{t}\frac{\ud s}{s}    \\
		                 & = \int_{ Q} f(x) \theta_z(x) \ud x=\int_{ Q} f(x) (\theta_z(x) - \theta_z(c_Q))\ud x,
	\end{align*}
	where in the last step we have finally used the critical assumption $\int_Q f = 0$.
	Here $c_Q$ is the center of $Q$.
	Since
	\begin{align*}
		|\theta_z(x) - \theta_z(c_Q)|
		 & = \Bigg|\frac{1}{\beta (t_{x,z}^{\beta-1}- (-s_{x,z})^{\beta-1})   t_{x,z}s_{x,z}}
		\Bigg( 1- \frac{ (t_{x,z}^{\beta-1}- (-s_{x,z})^{\beta-1}) t_{x,z}s_{x,z}}
		{ (t_{c_Q,z}^{\beta-1}- (-s_{c_Q,z})^{\beta-1})   t_{c_Q,z}s_{c_Q,z}}\Bigg)\Bigg|     \\
		 & \lesssim (A_1\ell(I))^{-(\beta+1)}
		\Bigg| 1- \frac{\big((A_2+3)^{\beta-1}-(A_1-2)^{\beta-1}\big)(A_2+3)(A_1+2)}
		{\big((A_2-3)^{\beta-1}-(A_1+2)^{\beta-1}\big)(A_2-3)(A_1-2)}\Bigg|                   \\
		 & \le  (A_1\ell(I))^{-(\beta+1)}\varepsilon,
	\end{align*}
	provided $A_1$ is big enough, combining the above and \eqref{eq:110} we get
	\begin{align*}
		|\widetilde f_P(z)- D_{z}f| & \lesssim A_1^{\beta+1}  (A_1\ell(I))^{-(\beta+1)} \varepsilon \|f\|_\infty |Q| =  \varepsilon\|f\|_\infty.
	\end{align*}
	This completes the proof.
\end{proof}

\begin{proof}[Proof of Theorem \ref{thm:main}]
	Fix $Q = I \times J \in \mathcal R_{\gamma}$
	and then choose $f$ with $\|f\|_{L^{\infty}} \lesssim 1$ and $\int_Q f = 0$ so that
	\[
		\int_Q|b- \langle b\rangle_Q| = \int b f.
	\]
	Now, apply our approximate weak factorization, Proposition \ref{prop:awf},
	to get the said decomposition of $f$.
	Then, since $(Q, W_1)$ and $(P, W_2)$ are in symmetric roles,
	we also have the following factorization for $\widetilde f_P$:
	\[
		\widetilde f_P = [u_P H_\gamma g_{W_2}- g_{W_2}H_\gamma^* u_P]+
		[u_{W_2}H_\gamma^* g_Q- g_Q H_\gamma u_{W_2}]+ \widetilde f_Q,
	\]
	where
	\[
		u_P= \frac { \widetilde f_P}{H_\gamma g_{W_2}},\quad
		u_{W_2}= \frac{g_{W_2} H_\gamma^* u_P}{ H_\gamma^* g_Q},\quad
		\widetilde f_Q= g_Q H_\gamma \Bigg( \frac{g_{W_2}}{H_\gamma^* g_Q}
		H_\gamma^* \Big(\frac { \widetilde f_P}{H_\gamma g_{W_2}}\Big)\Bigg)
	\]
	and there holds that
	\[
		|u_P|\lesssim A_1 | \widetilde f_P| \lesssim A_1 1_P,\qquad
		|u_{W_2}|\lesssim A_1 \|  \widetilde f_P\|_\infty 1_{W_2}\lesssim A_1  1_{W_2}.
	\]
	Moreover, given $\epsilon > 0$, for all $A_1$ large enough (independently of $Q$), there holds that
	\[
		\int_{Q} \widetilde f_Q =0,\qquad |\widetilde f_Q|
		\lesssim \varepsilon \|  \widetilde f_P\|_\infty 1_Q
		\lesssim \varepsilon^2 1_Q .
	\]
	By our decomposition, we have
	\begin{align*}
		\int bf & =-  \int g_{W_1}[b, H_\gamma](h_Q)+ \int h_{W_1}  [b, H_\gamma](g_P) + \int u_P[b, H_\gamma](g_{W_2}) \\
		        & \quad- \int  g_Q [b, H_\gamma](u_{W_2})+ \int b \widetilde f_Q.
	\end{align*}
	By the vanishing moment of $ \widetilde f_Q$ we have
	\[
		\int b \widetilde f_Q= \int_Q (b-\langle b\rangle_Q) \widetilde f_Q.
	\]
	Hence, we get
	\begin{align*}
		\int_Q|b- \langle b\rangle_Q| & \le \| [b, H_\gamma] \|_{L^p\to L^p}\big( \| g_{W_1}\|_{L^{p'}}
		\| h_Q\|_{L^p} +\| h_{W_1}\|_{L^{p'}} \| g_P\|_{L^p} +
		\| g_{W_2}\|_{L^{p}} \| u_P\|_{L^{p'}}                                                          \\
		                              & \qquad+ \| u_{W_2}\|_{L^{p}} \| g_Q\|_{L^{p'}}\big)
		+ C \varepsilon^2\int_Q|b- \langle b\rangle_Q|                                                  \\
		                              & \le C_{A_1}\| [b, H_\gamma] \|_{L^p\to L^p}|Q|
		+ C \varepsilon^2\int_Q|b- \langle b\rangle_Q|,
	\end{align*}
	where $C$ is absolute constant and we have, in addition to the various size estimates of the appearing
	functions, used Lemma \ref{lem:sizew}. Hence, by fixing $\varepsilon$ small enough (that is $A_1$ big enough)
	we obtain
	\[
		\frac 1{|Q|}\int_Q|b- \langle b\rangle_Q|\lesssim \| [b, H_\gamma] \|_{L^p\to L^p}.
	\]
	This completes the proof.
\end{proof}

\section{Curves with non-vanishing torsion}\label{sec:nonvan}
In this section we consider curves with non-vanishing torsion. In particular, we are interested in $C^2$ curves $\gamma: [-1,1] \to \R^2$ with a Taylor expansion of the form
\[
	\gamma(t)= \gamma'(0) t+ \frac{\gamma''(0)}{2} t^2 + R(t), \qquad t\in (-1,1),
\]
where $\gamma'(0), \gamma''(0)$ are linearly independent
and the remainder term $R$ is assumed to satisfy
\begin{align}\label{eq:error}
	R\in C^2(-1,1),\qquad \lim_{|t|\to 0} \frac{|R(t)|}{|t|^2} = 0.
\end{align}

Writing $\gamma'(0),  \gamma''(0)$
as column vectors we define the matrix
\[
	A:=
	\begin{bmatrix}
		\gamma'(0) & \frac{\gamma''(0)}2
	\end{bmatrix}.
\]
Then a straightforward computation shows that the matrix $O,$
\begin{equation}\label{eq:matrixo}
	O:= A B,\qquad B :=  \begin{bmatrix}
		|\gamma'(0)|^{-1} & -\Big\langle \frac{\gamma''(0)}2, \frac{\gamma'(0)}{|\gamma'(0)|}
		\Big\rangle |\gamma'(0)|^{-1} c                                                       \\ 0 & c\end{bmatrix},
\end{equation}
is orthonormal, where
\[
	c := \Bigg|\frac{\gamma''(0)}2- \Big\langle \frac{\gamma''(0)}2, \frac{\gamma'(0)}{|\gamma'(0)|}
	\Big\rangle\frac{\gamma'(0)}{|\gamma'(0)|}\Bigg|^{-1}.
\]

Denote the columns of $O$ by $e_1, e_2$. As in \cite{ClOu}, we let $\mathcal R_\gamma$
be the collection of rectangles with sides parallel to $e_1, e_2$
and satisfying that, for some $\ell$, the side-length of the side parallel to $e_j$ is $\ell^j$, $j=1,2$.
We say $b\in \BMO_\gamma(\R^2)$ if
\begin{align*}
	\|b\|_{\BMO_\gamma(\R^2)} :=
	\sup_{\substack{Q\in \mathcal R_\gamma \\ \ell(Q)\le 1} }\frac 1{|Q|}\int_Q|b- \langle b\rangle_Q|,
\end{align*}
where $\ell(Q)$ denotes the side-length of the side parallel to $e_1$.
It is stated in \cite{BGLW} that if $b\in \BMO_\gamma(\R^2)$,
then $[b, H_\gamma]$ is bounded in $L^p$, where
\[
	H_\gamma f(x) := \int_{-1}^1 f(x-\gamma(t))\frac{\ud t}t.
\]

Let $\mathcal P$ be the collection of all the parabolic cubes in $\R^2$ --
recall that $Q=I\times J$ is a parabolic cube if the sides of $Q$ are parallel
to the coordinate axes and $\ell(J)=\ell(I)^2$. For a parabolic cube we define $\ell(Q)=\ell(I)$.
Then it is clear that $\mathcal R_\gamma= O\mathcal P$.

We next record several small auxiliary results for the various appearing $\BMO$ spaces.
\begin{lem}\label{lem:small}
	Let $E$ be an invertible $2\times 2$ matrix.
	Let $1\le p<\infty$ and  $0<\tau<1$. Then we have that
	\begin{equation}\label{eq:2424}
		\sup_{\substack{Q\in \mathcal P\\ \ell(Q)\le 1} }
		\Big(\frac 1{|EQ|}\int_{EQ}|b- \langle b\rangle_{EQ}|^p\Big)^{1/p}
		\lesssim_{p,\tau} \sup_{\substack{Q\in \mathcal P\\ \ell(Q)\le \tau} }
		\Big(\frac 1{|EQ|}\int_{EQ}|b- \langle b\rangle_{EQ}|^p\Big)^{1/p},
	\end{equation}
	where the implicit constant is independent from $E$.
\end{lem}
\begin{proof}  By the change of variables
	$
		\int_{EQ}f  = |\det E|\int_Q f\circ E$ and $|EQ| = |\det E||Q|
	$
	it is enough to prove claim with  $E=\id$.
	But for the collection of parabolic cubes $\calP$ the claim
	is proved exactly as in the Euclidean case, which is simple.
\end{proof}
\begin{rem}
	Following the same scheme we also record the following estimates, which will be needed later:
	\begin{equation}\label{eq:26422}
		\sup_{\substack{Q\in \mathcal P\\ \ell(Q)\le \Lambda \tau} }
		\Big(\frac 1{|EQ|}\int_{EQ}|b- \langle b\rangle_{EQ}|^p\Big)^{1/p}
		\lesssim_{p,\Lambda} \sup_{\substack{Q\in \mathcal P\\ \ell(Q)\le \tau} }
		\Big(\frac 1{|EQ|}\int_{EQ}|b- \langle b\rangle_{EQ}|^p\Big)^{1/p},
	\end{equation}
	where $\Lambda\ge 1$ and the implicit constant is independent from $E$ and $\tau$.
\end{rem}

We also make the following observation that removes the matrix $B$ from the rest of our considerations.
\begin{lem}\label{lem:equibmo}
	Let $1\le p<\infty$. Then
	\[
		\sup_{Q\in \mathcal P\,: \ell(Q)\le 1} \Big(\frac 1{|AQ|}\int_{AQ}
		|b- \langle b\rangle_{AQ}|^p\Big)^{1/p}
		\sim \sup_{R\in \mathcal R_\gamma:\, \ell(R)\le 1}
		\Big(\frac 1{|R|}\int_{R}|b- \langle b\rangle_{R}|^p\Big)^{1/p}.
	\]
\end{lem}

\begin{proof}
	We begin with the observation that $\mathcal R_\gamma=\{ABQ: Q\in \mathcal P\}$,
	where $B$ is defined in \eqref{eq:matrixo}. Again,
	a change of variables gives that $\int_{AQ}b  = |\det A|\int_Q b\circ A$ and $|AQ| = |\det A||Q|.$
	Hence, it suffices to connect the quantities
	\[
		\sup_{Q\in \mathcal P: \ell(Q)\le 1}
		\Big(\frac 1{|Q|}\int_Q|b\circ A- \langle b\circ A\rangle_Q|^p\Big)^{1/p}
	\]
	and
	\[
		\sup_{Q\in \mathcal P: \ell(Q)\le 1}
		\Big(\frac 1{|BQ|}\int_{BQ}|b\circ A- \langle b\circ A\rangle_{BQ}|^p\Big)^{1/p}
	\]
	by suitable  coverings.
	Suppose the lower left corner of $Q$ is $(u_1, u_2)$. Then $B$ maps $ (u_1, u_2)+ \ell(I) (1,0)$ into
	\[
		B \left(
		\begin{array}{c}
				u_1 \\
				u_2 \\
			\end{array}
		\right)+ \ell(I)B \left(
		\begin{array}{c}
				1 \\
				0 \\
			\end{array}
		\right)= B \left(
		\begin{array}{c}
				u_1 \\
				u_2 \\
			\end{array}
		\right)+ |\gamma'(0)|^{-1} \ell(I)\left(
		\begin{array}{c}
				1 \\
				0 \\
			\end{array}
		\right).
	\]
	Also, $B$ maps $ (u_1, u_2)+ \ell(I)^2 (0,1)$ into
	\[
		B \left(
		\begin{array}{c}
				u_1 \\
				u_2 \\
			\end{array}
		\right)+ \ell(I)^2 B \left(
		\begin{array}{c}
				0 \\
				1 \\
			\end{array}
		\right)= B \left(
		\begin{array}{c}
				u_1 \\
				u_2 \\
			\end{array}
		\right)+ c \ell(I)^2  \left(
		\begin{array}{c}
				-\Big\langle \frac{\gamma''(0)}2, \frac{\gamma'(0)}{|\gamma'(0)|}\Big\rangle |\gamma'(0)|^{-1} \\
				1                                                                                              \\
			\end{array}
		\right).
	\]
	So we have seen that $B$ maps $Q$ into a parallelogram with one side parallel to the axis.
	When $\ell(Q)$ is sufficiently small, it is clear that
	there is a parabolic cube $Q_1$ such that $ BQ\subset Q_1$ and $\ell(Q_1) \sim \ell(Q)$.
	Hence, we have
	\begin{align*}
		\Big(\frac 1{|BQ|}\int_{BQ}|b\circ A- \langle b\circ A\rangle_{BQ}|^p\Big)^{1/p}
		 & \le 2\Big(\frac 1{|BQ|}\int_{BQ}|b\circ A- \langle b\circ A\rangle_{Q_1}|^p\Big)^{1/p}        \\
		 & \lesssim \Big(\frac 1{|Q_1|}\int_{Q_1}|b\circ A- \langle b\circ A\rangle_{Q_1}|^p\Big)^{1/p}.
	\end{align*}

	Now, $B^{-1}$ is also an upper triangular matrix.
	When $\ell(Q)$ is sufficiently small, there is a parabolic cube $Q_2$ such that
	$Q\subset BQ_2$ (i.e. $B^{-1}Q\subset Q_2$) and $\ell(Q_2) \sim \ell(Q)$, and so
	\[
		\Big(\frac 1{|Q|}\int_{Q}|b\circ A- \langle b\circ A\rangle_{Q}|^p\Big)^{1/p}
		\lesssim \Big(\frac 1{|BQ_2|}\int_{BQ_2}|b\circ A- \langle b\circ A\rangle_{BQ_2}|^p\Big)^{1/p}.
	\]
	Thus, the desired equivalence is proved for $\ell(Q)$ small enough.
	The general case is obtained by applying Lemma \ref{lem:small}.
\end{proof}

Finally, we note the John-Nirenberg property of $\BMO_\gamma(\R^2)$.
\begin{lem}\label{lem:JNP}
	Suppose that $b\in \BMO_\gamma(\R^2)$. Then for any $1<p<\infty$ and $0<\tau<1$ we have
	\[
		\sup_{\substack{Q\in \mathcal R_\gamma\\ \ell(Q)\le \tau} }
		\Big(\frac 1{|Q|}\int_Q |b\circ A- \langle b\circ A\rangle_Q |^p\Big)^{\frac 1p}\lesssim \sup_{\substack{Q\in \mathcal P\\ \ell(Q)\le \tau}}\frac 1{|Q|}\int_Q |b\circ A- \langle b\circ A\rangle_Q |.
	\]
\end{lem}
\begin{proof}
	By a change of variables the proof is essentially as in the classical Euclidean case.
\end{proof}

Now, a crucial observation is that to estimate
$\|b\|_{\BMO_\gamma(\R^2)}$, by Lemma \ref{lem:small}, Lemma \ref{lem:equibmo}
and a change of variables we know that it is enough to bound
\[
	\sup_{\substack{Q\in \mathcal P\\ \ell(Q)\le \tau}}\frac 1{|Q|}\int_Q |b\circ A- \langle b\circ A\rangle_Q |,
\]
where $\tau$ is a sufficiently small number to be specified later.
Fix $Q$. By applying the result in \cite{OIKARI-PARABOLIC} we get that
\begin{align*}
	\frac 1{|Q|}\int_Q |b\circ A- \langle b\circ A\rangle_Q |
	 & \lesssim \frac 1{|Q|}\sum_{j=1}^4 \left|\int \psi_j [b\circ A, H_\eta] \varphi_j\right|,
\end{align*}
where $\eta(t) := (t, t^2)$, $(\supp \psi_j, \supp \varphi_j)\in \{(Q, W), (P, W), (W, Q), (W, P)\}.$
Here
\[
	|\psi_j|\lesssim 1,\quad |\varphi_j|\lesssim 1,\quad |Q|\sim |W|\sim |P|,
\]
and
\[
	|x_1-y_1| \sim \ell(Q), \quad\forall x\in \supp(\varphi_j)\forall y\in \supp(\psi_j),
\]
which in particular gives
\[
	\dist (\pi_1 P, \pi_1Q)\sim \dist (\pi_1 P, \pi_1W)\sim \dist (\pi_1 Q, \pi_1 W)\sim \ell(Q).
\]
Here $W$ is the key auxiliary set defined in \cite{OIKARI-PARABOLIC} which is different from the $W_1,W_2$ defined in this article.
We have
\begin{align*}
	\int \psi_j [b\circ A, H_\eta] \varphi_j
	 & = \int \psi_j(x)\int \big(b(Ax)- b(Ax- A\eta (t))\big)\varphi_j(x-\eta(t))\frac{\ud t}{t}\ud x \\
	 & = |\det A|^{-1} \int \psi_j(A^{-1}x)\int \big(b(x)- b(x-\wt \gamma (t))\big)
	\varphi_j(A^{-1}x-\eta(t))\frac{\ud t}{t}\ud x                                                    \\
	 & = |\det A|^{-1} \int (\psi_j \circ A^{-1}) [b, H_{\wt \gamma}] (\varphi_j \circ A^{-1}),
\end{align*}
where $\wt \gamma= A\eta$ and we also note that in the above integral the support localization of the pair $(\supp \psi_j, \supp \varphi_j)$ forces
$ |t| \sim   \ell(Q)$.

We are ready to state and prove the main result of this section -- the current proof requires
the a priori qualitative assumption $b \in \BMO_{\gamma}(\R^2)$ that we did not attempt to
remove. A possible way to remove the a priori assumption would be to repeat proofs given in Section \ref{sect:pain}
for curves with non-vanishing torsion.
\begin{thm}
	Suppose $1 < p < \infty$ and $\|[b, H_\gamma]\|_{L^p\to L^p}<\infty$. With the qualitative a priori assumption
	$b\in \BMO_\gamma(\R^2)$ we have the quantitative bound
	\[
		\| b\|_{\BMO_\gamma(\R^2)}\lesssim \| [b, H_{ \gamma}] \|_{L^p\to L^p}.
	\]
\end{thm}

\begin{proof}
	By the above discussion, for fixed $Q\in \mathcal P$ with $\ell(Q)\le \tau$ we have
	\begin{align*}
		\frac 1{|Q|}\int_Q |b\circ A- \langle b\circ A\rangle_Q |
		 & \lesssim \sum_{j=1}^4\frac 1{|\det A||Q|}\left| \int (\psi_j \circ A^{-1})
		[b, H_{\wt \gamma}] (\varphi_j \circ A^{-1})\right|                           \\
		 & \le \sum_{j=1}^4\frac 1{|\det A||Q|}\left| \int (\psi_j \circ A^{-1})
		[b, H_{  \gamma}] (\varphi_j \circ A^{-1})\right|                             \\
		 & + \sum_{j=1}^4\frac 1{|\det A||Q|}\left| \int (\psi_j \circ A^{-1})
		[b, H_{  \wt \gamma}-H_{   \gamma} ] (\varphi_j \circ A^{-1})\right|          \\
		 & =: I+II.
	\end{align*}
	The estimate of $I$ is easy -- indeed, we have
	\begin{align*}
		I & \lesssim \sum_{j=1}^4\frac 1{|\det A||Q|}\| \psi_j \circ A^{-1}\|_{L^{p'}}
		\| [b, H_{\gamma}]\|_{L^p\to L^p}\|\varphi_j \circ A^{-1}\|_{L^p}
		\lesssim \| [b, H_{\gamma}]\|_{L^p\to L^p}.
	\end{align*}

	We then focus on the term $II$. Recall that the supports of $\psi_j$ and $\varphi_j$ force
	$|t|\sim \ell(Q)$ in the kernel representation of
	\[
		\int (\psi_j \circ A^{-1}) [b, H_{  \wt \gamma} ] (\varphi_j \circ A^{-1}).
	\]
	Next, we look at
	\begin{align*}
		\int (\psi_j \circ A^{-1}) & [b, H_{   \gamma} ] (\varphi_j \circ A^{-1}) \\
		                           & =\int \psi_j (A^{-1} x)
		\int \big(b(x)- b(x-  \gamma (t))\big)
		\varphi_j(A^{-1}x-A^{-1}\gamma(t))\frac{\ud t}{t}\ud x.
	\end{align*}
	Again,
	we have $ |A^{-1}\gamma(t)|\sim \ell(Q)$. Now, observe that if $|t| \ll \ell(Q)$, then by $\gamma$ being $C^2$ (e.g. $C^1$ with bounded derivative)
	we have
	\[
		|A^{-1} \gamma(t)|\le C|\gamma(t)|=C|\gamma(t)-\gamma(0)|\le C' |t|\ll \ell(Q),
	\]
	which is a contradiction. Hence $|t|\gtrsim \ell(Q)$. On the other hand, by the limiting assumption on the line \eqref{eq:error}
	we may let $c$ be an absolute constant (depending only on $\gamma$) such that
	\[
		|\gamma(t)-\wt \gamma(t)|\le |R(t)| \leq \frac 12 |\wt \gamma(t)|,\quad \forall |t|\in [0,c].
	\]
	We used above that $ |\wt \gamma(t)| \sim |t|.$
	Thus, for $|t|\in [0,c]$ we have $|\gamma(t)|\sim |\wt \gamma(t)|$,
	from which we get that $|t|\sim \ell(Q)$ (assuming $\tau/c$ is sufficiently small).

	Hence, by the above discussion there exits $\theta, \Theta$ such that
	\begin{align*}
		II & =  \sum_{j=1}^4\frac 1{|\det A||Q|}\left| \int (\psi_j \circ A^{-1})
		[b, H_{\theta\ell(Q), \Theta \ell(Q)}^{\wt \gamma }
		-H_{\theta\ell(Q), \Theta \ell(Q)}^{  \gamma} ] (\varphi_j \circ A^{-1})\right| \\
		   & + \sum_{j=1}^4\frac 1{|\det A||Q|}\int_{\R^2} \int_{|t|\in (c,1]}
		\big| \psi_j (A^{-1} x)  \big(b(x)- b(x-  \gamma (t))\big)\varphi_j(A^{-1}x-A^{-1}\gamma(t))\big|
		\frac{\ud t}{t}\ud x                                                            \\
		   & =:II_1+II_2,
	\end{align*}
	where both $ H_{\theta\ell(Q), \Theta \ell(Q)}^{ \wt \gamma}$ and $ H_{\theta\ell(Q), \Theta \ell(Q)}^{ \gamma}$ are defined through the formula
	\[
		H_{\theta\ell(Q), \Theta \ell(Q)}^{  \rho}f(x)
		:=\int_{ \theta\ell(Q)\le |t|\le \Theta \ell(Q)}f(x-\rho(t)) \frac{\ud t}{t}.
	\]

	We estimate $II_2$ first. A key observation is that
	\begin{equation}\label{eq:ec1}
		\lim_{\ell(Q)\to 0}|I_Q|=0,\qquad \text{where\,\,}
		I_Q:= ([-1,-c]\cup[c,1])\cap\{t: |A^{-1}\gamma(t)|\sim \ell(Q)\}.
	\end{equation}
	We will prove this soon. First, note that
	\begin{align*}
		 & II_2                                                                                    \\
		 & = \sum_{j=1}^4\frac 1{|Q|}\int_{\R^2} \int_{|t|\in (c,1]} \big|
		\psi_j ( x)  \big((b\circ A)(x)- (b\circ A)(x- A^{-1} \gamma (t))\big)
		\varphi_j(x-A^{-1}\gamma(t))\big|\frac{\ud t}{t}\ud x                                      \\
		 & \lesssim \sum_{j=1}^4 \frac 1{|Q|}\int_{I_Q}\int_{\R^2 }
		\big|(b\circ A)(x)- \langle b\circ A\rangle_{\Lambda Q}\big|
		|\psi_j ( x)|\cdot |\varphi_j(x-A^{-1}\gamma(t)) |\ud x \ud t                              \\
		 & \,\,+ \sum_{j=1}^4\frac 1{|Q|}\int_{I_Q}\int_{\R^2 }\big|(b\circ A)(x-A^{-1}\gamma(t))-
		\langle b\circ A\rangle_{\Lambda Q}\big| |\psi_j ( x)|
		\cdot |\varphi_j(x-A^{-1}\gamma(t)) |\ud x \ud t                                           \\
		 & \lesssim  \frac 1{|Q|}\int_{I_Q}\int_{\Lambda Q}\big|(b\circ A)(x)-
		\langle b\circ A\rangle_{\Lambda Q}\big|\ud x \ud t                                        \\
		 & \lesssim  |I_Q| \sup_{\substack{R\in \calP                                              \\ \ell(R)\le \Lambda\tau}}\frac 1{|R|}\int_R |b\circ A- \langle b\circ A\rangle_R |\\
		 & \lesssim  |I_Q| \sup_{\substack{R\in \calP                                              \\ \ell(R)\le  \tau}}\frac 1{|R|}\int_R |b\circ A- \langle b\circ A\rangle_R |,
	\end{align*}
	where we have used \eqref{eq:26422} in the last step and $ \Lambda>0$ is a suitable large constant such that $\Lambda Q$ is a parabolic cube
	with $\ell(\Lambda Q)=\Lambda \ell(Q)$ and $\Lambda Q \supset Q\cup W\cup P$.

	To complete the estimate of $II_2$ it remains to prove \eqref{eq:ec1}.
	Aiming for a contradiction, assume that
	\[
		\mathop{\lim\sup}_{\ell(Q)\to 0}|I_Q|>0.
	\]
	Then there exists some $\varepsilon>0$ and a sequence of parabolic cubes $\{Q_j\}$
	such that $\ell(Q_j)\to 0$ and $|I_{Q_j}|\ge \varepsilon$.
	Note that
	$
		|A^{-1}\gamma(t)|\sim |\gamma(t)|.
	$
	Thus, there exist constants  $c_1, c_2$ such that
	\[
		I_{Q_j}\subset ([-1,-c]\cup [c, 1])\cap \{t: c_1\ell(Q_j)
		\le  |\gamma(t)|\le c_2 \ell(Q_j)\},\quad \forall\, j.
	\]
	We may then further pick a subsequence of $\{Q_j\}$, which we denote by $Q_{j_k}$,
	such that $[c_1\ell(Q_{j_k}), c_2\ell(Q_{j_k})]$ are pairwise disjoint.
	Then $   I_{Q_j}$
	are pairwise disjoint and we obtain a contradiction
	\[
		1-c \ge \sum_k |I_{Q_{j_k}}|=+\infty.
	\]
	Thus, \eqref{eq:ec1} holds and for any prescribed $\varepsilon>0$ we may choose $\tau = \tau(\varepsilon)$ sufficiently small so that
	\[
		II_2 \le \varepsilon \sup_{\substack{R\in \calP\\ \ell(R)\le  \tau}}\frac 1{|R|}\int_R |b\circ A- \langle b\circ A\rangle_R |.
	\]

	The rest of the proof is devoted to estimating $II_1$. Write
	\begin{align*}
		II_1 & = \sum_{j=1}^4\frac 1{ |Q|}\left| \int  \psi_j
		[b\circ A, H_{\theta\ell(Q), \Theta \ell(Q)}^{A^{-1}\wt \gamma }
		-H_{\theta\ell(Q), \Theta \ell(Q)}^{  A^{-1}\gamma} ]  \varphi_j  \right|                                 \\
		     & = \sum_{j=1}^4\frac 1{ |Q|}\left| \int  \psi_j   [(b\circ A)- \langle b\circ A\rangle_{\Lambda Q},
		H_{\theta\ell(Q), \Theta \ell(Q)}^{A^{-1}\wt \gamma }
		-H_{\theta\ell(Q), \Theta \ell(Q)}^{  A^{-1}\gamma} ]  \varphi_j  \right|                                 \\
		     & \le \sum_{j=1}^4\frac 1{ |Q|}\Big(\big\| \psi_j \big((b\circ A)
		- \langle b\circ A\rangle_{\Lambda Q}\big)\big\|_{L^2} \big\|
		\big(H_{\theta\ell(Q), \Theta \ell(Q)}^{A^{-1}\wt \gamma }
		-H_{\theta\ell(Q), \Theta \ell(Q)}^{  A^{-1}\gamma} \big) \varphi_j \big\|_{L^2}                          \\
		     & \hspace{2cm}+\| \psi_j\|_{L^2}\big\| \big(H_{\theta\ell(Q), \Theta \ell(Q)}^{A^{-1}\wt \gamma }
		-H_{\theta\ell(Q), \Theta \ell(Q)}^{  A^{-1}\gamma} \big) \big(((b\circ A)
		- \langle b\circ A\rangle_{\Lambda Q})\varphi_j \big)\big\|_{L^2} \Big)                                   \\
		     & \lesssim \| H_{\theta\ell(Q), \Theta \ell(Q)}^{A^{-1}\wt \gamma }
		-H_{\theta\ell(Q), \Theta \ell(Q)}^{  A^{-1}\gamma}\|_{L^2}\Big(\fint_{\Lambda Q}
		|b\circ A- \langle b\circ A\rangle_{\Lambda Q} |^2\Big)^{1/2}                                             \\
		     & \lesssim \| H_{\theta\ell(Q), \Theta \ell(Q)}^{A^{-1}\wt \gamma }
		-H_{\theta\ell(Q), \Theta \ell(Q)}^{  A^{-1}\gamma}\|_{L^2} \sup_{\substack{R\in \calP                    \\ \ell(R)\le  \tau}}\frac 1{|R|}\int_R |b\circ A- \langle b\circ A\rangle_R |,
	\end{align*}where we have used \eqref{eq:26422} and Lemma \ref{lem:JNP} in the last step.
	Clearly, now the proof is completed by applying the following Lemma \ref{lem:ApplyCorput}. Indeed, by that lemma, for any $\eps>0$, we may let the above $\tau =\tau (\eps)$ be so small such that
	\[
		II_1 \le \eps \sup_{\substack{R\in \calP\\ \ell(R)\le  \tau}}\frac 1{|R|}\int_R |b\circ A- \langle b\circ A\rangle_R |.
	\]
	Then combining all the estimates together we finally have
	\begin{align*}
		\sup_{\substack{R\in \calP              \\ \ell(R)\le  \tau}}\frac 1{|R|}\int_R |b\circ A- \langle b\circ A\rangle_R |
		 & \le 2\eps \sup_{\substack{R\in \calP \\ \ell(R)\le  \tau}}\frac 1{|R|}\int_R |b\circ A- \langle b\circ A\rangle_R |+C_{\gamma,p}  \| [b, H_{\gamma}]\|_{L^p\to L^p}.
	\end{align*}
	Hence, by taking $\eps= 1/4$ and invoking Lemma \ref{lem:small} and Lemma \ref{lem:equibmo} we get the desired estimate.
\end{proof}

\begin{lem}\label{lem:ApplyCorput}
	There holds that
	\[
		\lim_{\ell(Q)\to 0} \| H_{\theta\ell(Q), \Theta \ell(Q)}^{A^{-1}\wt \gamma }-H_{\theta\ell(Q), \Theta \ell(Q)}^{  A^{-1}\gamma}\|_{L^2\to L^2} =0.
	\]
\end{lem}
\begin{proof}
	Fix some $f\in L^2$. By Plancherel's theorem we have
	\begin{align*}
		 & \| H_{\theta\ell(Q), \Theta \ell(Q)}^{A^{-1}\wt \gamma }f-H_{\theta\ell(Q), \Theta \ell(Q)}^{  A^{-1}\gamma}f\|_{L^2}
		\\
		 & \hspace{2.5cm}= \Big\| \int_{ \theta\ell(Q)\le |t|\le \Theta \ell(Q)}\left(e^{-2\pi i \xi\cdot  A^{-1}\wt \gamma(t)} -e^{-2\pi i \xi\cdot  A^{-1}  \gamma(t)}\right)\frac{\ud t}{t} \widehat f(\xi)\Big\|_{L^2}.
	\end{align*}

	It suffices to provide uniform bounds in $\xi$ for
	\[
		K = K(\xi)=\left|\int_{ \theta\ell(Q)}^{\Theta \ell(Q)}\left(e^{-2\pi i \xi\cdot A^{-1} \wt \gamma(t)} -e^{-2\pi i \xi\cdot    A^{-1}\gamma(t)}\right)\frac{\ud t}{t}\right|,
	\]
	as the integration over $[-\Theta \ell(Q), -\theta\ell(Q)]$ is similar.

	For any prescribed $\varepsilon>0,$ if $\tau = \tau(\varepsilon)$ is sufficiently small, we always have the bound
	\begin{align*}
		K \lesssim \int_{ \theta\ell(Q)}^{\Theta \ell(Q)}|\xi| |A^{-1}R(t)|\frac{\ud t}{t}\lesssim |\xi|\sup_{t\sim\ell(Q)}|R(t)|\lesssim \varepsilon|\xi|\ell(Q)^2.
	\end{align*}
	This bound is clearly not always sufficient. To obtain a complementary bound we use van der Corput's lemma.

	Set
	$$
		\wt \phi(t)=-2\pi  \xi\cdot A^{-1} \wt \gamma(t)= -2\pi  (\xi_1 t+ \xi_2 t^2),\qquad
		\phi(t)= -2\pi  \xi\cdot  A^{-1}\gamma(t).
	$$
	Let us first assume that $|\xi_1|\le  |\xi_2|.$
	Note that $|\wt \phi''(t)|=4\pi |\xi_2|$ and by continuity of $\gamma'',\wt\gamma''$ and that $\gamma''(0) = \wt\gamma''(0) = (0,2),$ there exists some absolute constant $\delta_1$ such that for $|t|\le \delta_1$ (we may let $\tau<\delta_1/\Theta$) we have $
		|(A^{-1}\gamma)''(t)-(A^{-1}\wt\gamma)''(t)|< 1/2 $,
	which gives that
	\begin{align*}
		|\phi''(t)|\ge |\wt\phi''(t)| - |\wt \phi''(t)-\phi''(t)| & \ge 4\pi |\xi_2| - 2\pi|\xi|	|\gamma''(t)-\wt\gamma''(t)|  \\
		                                                          & \ge4\pi |\xi_2|-  \pi|\xi| \geq 2\pi|\xi_2| \gtrsim |\xi|.
	\end{align*}
	Thus, by van der Corput's lemma \ref{lem:VDCL}, see below, we obtain
	\[
		K \lesssim |\xi|^{-\frac 12}\ell(Q)^{-1}.
	\]
	Then let us look at the case $|\xi_1| \geq |\xi_2|.$ It is clear that $\wt \phi'(t)=-2\pi (\xi_1+ 2\xi_2 t) $ is monotone and then
	\[
		|\wt \phi'(t)|\ge  |\xi_1|\gtrsim |\xi|,\qquad \forall|t|< 1/4.
	\]
	Thus, applying van der Corput's lemma we have
	\[
		\left|\int_{ \theta\ell(Q)}^{\Theta \ell(Q)} e^{i \wt \phi(t)} \frac{\ud t}{t}\right| \lesssim |\xi|^{-1}\ell(Q)^{-1}.
	\]
	However, the monotonicity of $\phi'$ is unclear. Instead, write
	\begin{align*}
		\left|\int_{ \theta\ell(Q)}^{\Theta \ell(Q)} e^{i \phi(t)} \frac{\ud t}{t}\right|= \left|\int_{ \theta\ell(Q)}^{\Theta \ell(Q)} e^{i \wt\phi(t)} e^{- 2\pi i \xi \cdot A^{-1}R(t)}\frac{\ud t}{t}\right|=: \left|\int_{ \theta\ell(Q)}^{\Theta \ell(Q)} e^{i \wt\phi(t)} \psi(t) \ud t\right|.
	\end{align*}
	Now, $\wt \phi'(t)=-2\pi (\xi_1+ 2\xi_2 t) $ is monotone. And note that
	\[
		\sup_{t\in [\theta\ell(Q), \Theta \ell(Q)]}|\psi(t)|=\sup_{t\in [\theta\ell(Q), \Theta \ell(Q)]}\big|e^{- 2\pi i \xi \cdot A^{-1}R(t)} t^{-1}\big|\sim \ell(Q)^{-1}
	\]
	and
	\begin{align*}
		\int_{ \theta\ell(Q)}^{\Theta \ell(Q)} |\psi'(t)|\ud t\lesssim \int_{ \theta\ell(Q)}^{\Theta \ell(Q)}(t^{-2}+ |\xi| |R'(t)| t^{-1})\ud t\le \ell(Q)^{-1}+ |\xi|\ell(Q),
	\end{align*}
	where we simply used
	\[
		|R'(t)|= |\gamma'(t)-\wt \gamma'(t)|\le  |\gamma'(t)- \gamma'(0)|+|\wt \gamma'(t)- \wt \gamma'(0)|\lesssim t.
	\]
	Thus, by van der Corput's lemma  we get
	\[
		\left|\int_{ \theta\ell(Q)}^{\Theta \ell(Q)} e^{i\wt\phi(t)} \psi(t) \ud t\right|\lesssim |\xi|^{-1}(\ell(Q)^{-1}+ |\xi|\ell(Q))\le |\xi|^{-1}\ell(Q)^{-1}+\ell(Q).
	\]
	Gathering the above estimates together we have obtained
	\begin{align*}
		K & \lesssim \min\left(\varepsilon|\xi|\ell(Q)^2,\quad 1_{|\xi_1|\leq |\xi_2|}|\xi|^{-\frac 12}\ell(Q)^{-1} +   1_{|\xi_1|\geq |\xi_2|}\left(|\xi|^{-1}\ell(Q)^{-1}+\ell(Q)\right)\right).
	\end{align*}
	From the above estimate we obtain
	\begin{equation}
		K\lesssim
		\begin{cases}
			\varepsilon^{1/2},\qquad                                     & |\xi|\le \varepsilon^{-1/2}\ell(Q)^{-2}, \\
			\varepsilon^{1/4}+ \varepsilon^{1/2}\ell(Q) + \ell(Q),\qquad & |\xi|>  \varepsilon^{-1/2}\ell(Q)^{-2},
		\end{cases}
	\end{equation}
	which gives the desired asymptotics.
\end{proof}

\begin{lem}[Van der Corput's lemma]\label{lem:VDCL}
	Let $I = [a,b]$ be a fixed interval and $\phi,\psi:I\to \R.$ Denote
	\[
		I(a,b) = \int_a^b e^{i\phi(t)}\psi(t)\ud t.
	\]
	Then, there exist absolute constants $C,C_k>0$ such that the following hold.
	\begin{enumerate}
		\item If $|\phi'(t)|\geq \lambda > 0$ and $\phi'$ is monotone, then
		      \[
			      |I(a,b)|\leq C\lambda^{-1}\Big(\| \psi\|_{L^{\infty}(I)} + \|\psi'\|_{L^1(I)}\Big).
		      \]
		\item If $\phi\in C^k[a,b]$ and $|\phi^{(k)}(t)|>\lambda > 0,$ then
		      \[
			      |I(a,b)|\leq C_k\lambda^{-1/k}\Big(\| \psi\|_{L^{\infty}(I)} + \|\psi'\|_{L^1(I)}\Big).
		      \]
	\end{enumerate}
\end{lem}

\bibliography{refH}

@Article{	  BGLW,
  author	= {Bongers, Tyler and Guo, Zihua and Li, Ji and Wick, Brett},
  title		= {Commutators of {H}ilbert transforms along monomial
		  curves},
  year		= {2021},
  journal	= {Studia Math.},
  volume	= {257},
  pages		= {295-311},
  doi		= {10.4064/sm190915-22-4}
}

@Article{	  ClOu,
  author	= {Cladek, Laura and Ou, Yumeng},
  title		= {Sparse domination of {H}ilbert transforms along curves},
  year		= {2018},
  journal	= {Math. Res. Lett.},
  volume	= {25},
  pages		= {415--436}
}

@Article{	  CRW,
  author	= {Coifman, R. R. and Rochberg, R. and Weiss, Guido},
  title		= {Factorization theorems for {H}ardy spaces in several
		  variables},
  journal	= {Ann. of Math. (2)},
  fjournal	= {Annals of Mathematics. Second Series},
  volume	= {103},
  year		= {1976},
  number	= {3},
  pages		= {611--635},
  issn		= {0003-486X},
  mrclass	= {42A40},
  mrnumber	= {412721},
  mrreviewer	= {D. Sarason},
  doi		= {10.2307/1970954},
  url		= {https://doi.org/10.2307/1970954}
}

@Article{	  HyCom,
  title		= {The {$L^p$}-to-{$L^q$} boundedness of commutators with
		  applications to the {J}acobian operator},
  fjournal	= {Journal de Mathématiques Pures et Appliquées},
  journal	= {J. Math. Pures Appl.},
  volume	= {156},
  pages		= {351-391},
  year		= {2021},
  issn		= {0021-7824},
  doi		= {https://doi.org/10.1016/j.matpur.2021.10.007},
  url		= {https://www.sciencedirect.com/science/article/pii/S0021782421001495},
  author	= {Hyt\"{o}nen, T. P.}
}

@Article{	  OIKARI-PARABOLIC,
  title		= {Lower bound of the parabolic {H}ilbert commutator},
  journal	= {Adv. Math.},
  volume	= {404},
  pages		= {108451},
  year		= {2022},
  doi		= {10.1016/j.aim.2022.108451},
  author	= {Tuomas Oikari}
}
\end{document}